\documentclass[12pt]{article}
\usepackage[utf8]{inputenc}
\usepackage[english]{babel}
 
\usepackage{geometry}
\usepackage{bm}
\usepackage{bbm}
\usepackage{easybmat}
\usepackage{dsfont}
\usepackage{tabularx}
\usepackage{etex}
\usepackage{stmaryrd}
\usepackage{amsthm}
\usepackage{amsmath}
\usepackage{amsfonts}
\usepackage{amssymb}
\usepackage{graphicx}
\usepackage{array}
\usepackage{booktabs}
\usepackage[usenames]{color}
\usepackage{epstopdf}
\usepackage{caption}
\usepackage{subcaption}
\usepackage{float}
\usepackage{algorithm}
\usepackage{algpseudocode}
\usepackage{nomencl}
\usepackage[font={footnotesize}]{subcaption}
\usepackage[font={footnotesize, it}]{caption}
\usepackage{url}
\usepackage{xr}
\usepackage{setspace}
\usepackage{epstopdf}
\usepackage{enumitem}
\usepackage[round]{natbib}
\usepackage{enumitem}
\usepackage{hyperref}

\newtheorem{theorem}{Theorem}[section]

\newtheorem{corollary}[theorem]{Corollary}

\newtheorem{definition}[theorem]{Definition}

\newtheorem{lemma}[theorem]{Lemma}

\newtheorem{proposition}[theorem]{Proposition}
\newtheorem{remark}[theorem]{Remark}

\numberwithin{equation}{section}

\author{Micha{\l} Kos$^a$ and Ma{\l}gorzata Bogdan$^{a,b}$\\\\\\
\noindent $^a$ Department of Mathematics, University of Wroclaw,  Poland \\
\noindent $^b$ Department of Statistics, Lund University,  Sweden }

\title{On the asymptotic properties of SLOPE}

\begin{document}
\maketitle
\abstract{Sorted L-One Penalized Estimator (SLOPE) is a relatively new convex optimization procedure for selecting predictors in high dimensional regression analyses. SLOPE extends LASSO by replacing the $L_1$ penalty norm with a Sorted $L_1$ norm, based on the non-increasing sequence of tuning parameters. This allows SLOPE to adapt to unknown sparsity and achieve an asymptotic minimax convergency rate under a wide range of high dimensional generalized linear models. Additionally, in the case when the design matrix is orthogonal, SLOPE  with the sequence of tuning parameters $\lambda^{BH}$ corresponding to the sequence of decaying thresholds for the Benjamini-Hochberg multiple testing correction provably controls the False Discovery Rate (FDR) in the multiple regression model. In this article we provide new asymptotic results on the properties of SLOPE when the elements of the design matrix are iid random variables from the Gaussian distribution. Specifically, we provide conditions under which the asymptotic FDR  of SLOPE based on the sequence $\lambda^{BH}$  converges to zero and the power converges to 1. We illustrate our theoretical asymptotic results with an extensive simulation study. We also provide precise formulas describing FDR of SLOPE under different loss functions, which sets the stage for future investigation on the model selection properties of SLOPE and its extensions.} 
\section{Introduction}
\label{sec:Intro}

In this article we consider the classical problem of identifying important predictors in the multiple regression model
\begin{equation}\label{model}
Y=X b^0+\epsilon\;\;,
\end{equation}
where $X$ is the design matrix of dimension $n\times p$ and $\epsilon \sim N(0,\sigma^2 I_{n\times n})$ is the noise vector.
In the case when $p>n$, the vector of parameters $b^0$ is not identifiable and  can be uniquely estimated only under certain additional assumptions concerning e.g. its sparsity (i.e. the number of non-zero elements). One of the most popular and computationally tractable methods for estimating  $b^0$ in the case when $p>n$ is the Least Absolute Shrinkage and Selection Operator (LASSO, \cite{LASSO}), firstly introduced in the context of signal processing as the Basis Pursuit Denoising (BPDN, \cite{chen1994basis}). LASSO estimator is defined as 

\begin{equation}\label{LASSO}
\hat b^{L}=argmin_{b}  \left\{ \frac{1}{2} || Y-Xb||_2^2+\lambda |b| _{1}\right\}\;\;,
\end{equation}
where $||\cdot||_2$ denotes the regular Euclidean norm in $R^n$  and $|b|_1=\sum_{j=1}^p |b_j|$ is the $L_1$ norm of $b$. If $X'X=I$ then LASSO is reduced to the simple shrinkage operator imposed on the elements of the vector $\tilde Y = X'Y$,
$$b^L_i=\eta_{\lambda}(\tilde Y_i)\;\;,$$
where 
$$\eta_{\lambda}(x)=\left\{\begin{array}{ccc}
x-\lambda&\mbox{if}&x>\lambda\\
0&\mbox{if}&|x|<\lambda\\
x+\lambda&\mbox{if}&x<-\lambda
\end{array}
\right.\;\;.
$$
In this case the choice of the tuning parameter 
$$\lambda=\lambda^{Bon}=\sigma \Phi^{-1}\left(1-\frac{\alpha}{2p}\right)=\sigma \sqrt{2\log p}(1+o_p)\;\;,$$
allows for the control of the probability of at least one false discovery (Family Wise Error Rate, FWER) at level $\alpha$. 

In the context of high dimensional multiple testing, the Bonferroni correction is often replaced by the \cite{BH} multiple testing procedure aimed at the control of the False Discovery Rate (FDR). Apart from FDR control, this procedure also has appealing properties in the context of the  estimation of the vector of means for multivariate normal distribution with independent entries \citep{ABDJ} or in the context of minimizing the Bayes Risk related to 0-1 loss \citep{ABOS, ABOS2, ABOS3}. In the context of the multiple regression with an orthogonal design, the Benjamini-Hochberg procedure works as follows:

 \begin{itemize}
\item[(1)] Fix $q\in(0,1)$ and sort $\tilde Y$ such that
$$|\tilde Y |_{(1)} \ge |\tilde Y|_{(2)} \ge \ldots \ge |\tilde Y|_{(p)}\;\;,$$ 
\item[(2)] Identify the largest $j$ such that
\begin{equation}\label{BH}
|\tilde Y|_{(j)} \ge  \lambda^{BH}_j= \sigma \Phi^{-1}\left(1- \frac{jq}{2p}\right), 
\end{equation}
 Call this index
$j_{\text{BH}}$.
\item[(3)] Reject $H_{(j)}$ for every $j \leq j_{\text{BH}}$ \;\;.
\end{itemize}

Thus, in BH, the fixed threshold of the Bonferroni correction, $\lambda^{Bon}$, is replaced with the sequence 
$\lambda^{BH}$ of 'sloped' thresholds for the sorted test statistics (see Figure \ref{plot1}). This allows for a substantial increase of power and for an improvement of prediction properties in the case when some of the predictors are relatively weak.

\begin{figure}[h!]
        \centering
                \includegraphics[width=0.8\textwidth]{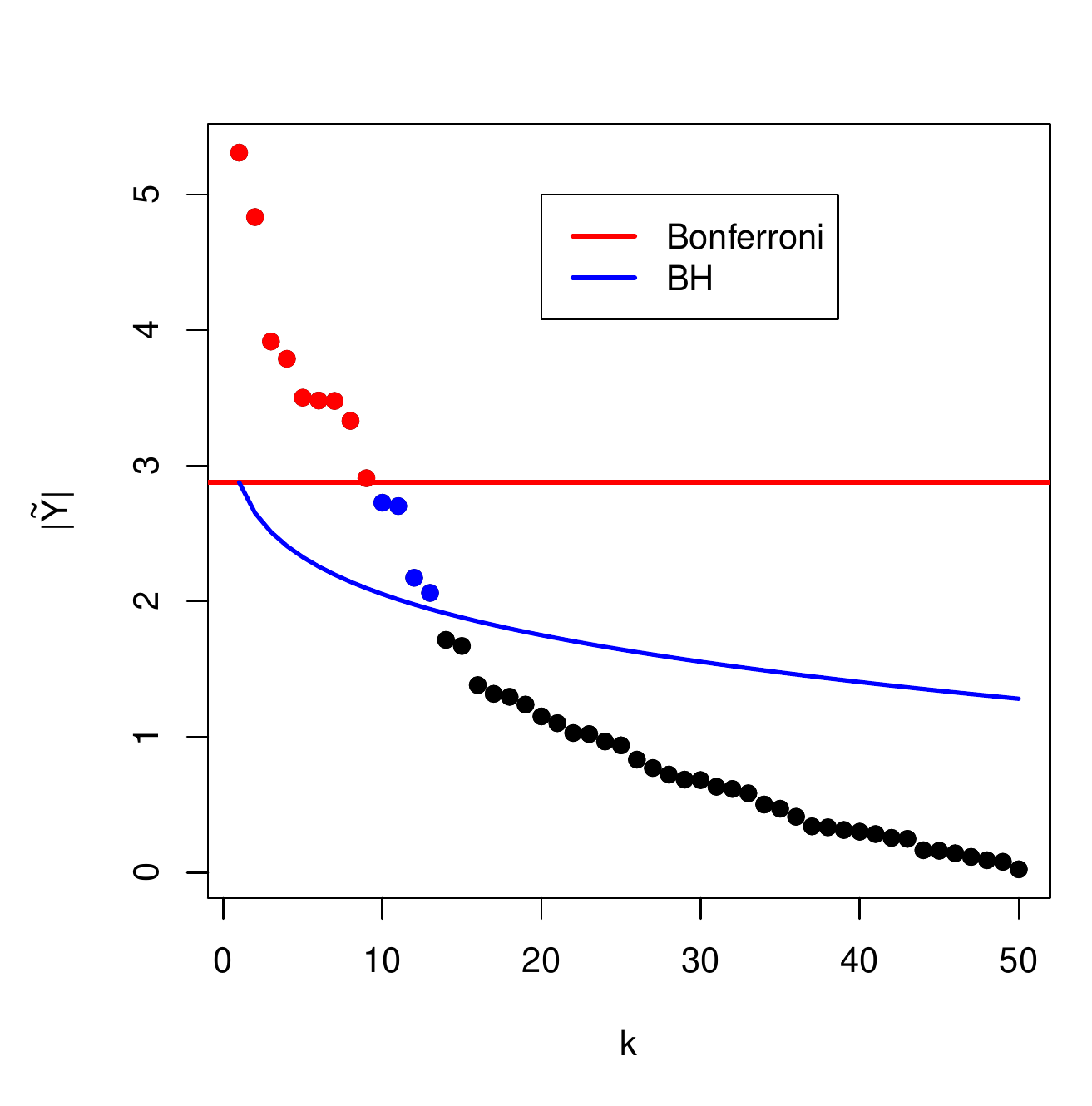}  
								\caption{Bonferroni and Benjamini-Hochberg procedures for multiple testing}
								\label{plot1}
														
 \end{figure}

The idea of using a decreasing sequence of thresholds was subsequently used in the Sorted L-One Penalized Estimator (SLOPE, \cite{SLOPE, SLOPE2})  for the estimation of coefficients in the multiple regression model:
\begin{equation}\label{SLOPE}
b^{SL}=argmin_b \left\{\frac 12\big\|y-Xb\big\|_2^2+\sum_{i=1}^p\lambda_i\big|b\big|_{(i)}\right\},
			\end{equation}
			where $\big|b\big|_{(1)}\geq\ldots\geq \big|b\big|_{(p)}$ are ordered magnitudes of elements of $b$
			and $\lambda=(\lambda_1,\ldots,\lambda_p)$ is a non-zero, non-increasing and non-negative sequence of tuning parameters.
 As noted in (\cite{SLOPE,SLOPE2}),  the function $J_{\lambda}(b)=\sum_{i=1}^p\lambda_i\big|b\big|_{(i)}$
is a norm. To see this, observe that:
\begin{itemize}
\item for any constant $a\in R$ and a sequence $b \in R^p$, $J_{\lambda}(a b)=|a|J_{\lambda}(b)$ 
\item  $J_{\lambda}(b)=0$ if and only if $b=0\in R^p$.
\end{itemize}

To show the triangular inequality: 
\begin{itemize}
\item $J_{\lambda}(x+y)\leq J_{\lambda}(x)+J_{\lambda}(y)$
\end{itemize}
let us denote by $\rho$ the permutation of the set $\{1,\ldots,p\}$ such that
$$\left|x_{\rho(1)}+y_{\rho(1)}\right|\geq \left|x_{\rho(2)}+y_{\rho(2)}\right|\geq ... \geq \left|x_{\rho(p)}+y_{\rho(p)}\right|\;\;.$$
Then
 \[
J_{\lambda}(x+y) = \sum_{i=1}^p\lambda_i\big|x_{\rho(i)}+y_{\rho(i)}\big| 
\leqslant \sum_{i=1}^p\lambda_i\big|x_{\rho(i)}\big|+\sum_{i=1}^p\lambda_i\big|y_{\rho(i)}\big|
\]
and the result can be derived from the well known rearrangement inequality, according to which 
for any permutation $\rho$ and any sequence $x\in R^p$:
\[\sum_{i=1}^p\lambda_i\big|x_{\rho(i)}\big| \leqslant \sum_{i=1}^p \lambda_i |x|_{(i)}\;\;. 
\]
Thus SLOPE is a convex optimization procedure which can be efficiently solved using classical optimization tools. 
It is also easy to observe that in the case $\lambda_1=\ldots=\lambda_p$, SLOPE is reduced to LASSO, while in the case   $\lambda_1>\lambda_2=\ldots=\lambda_p=0$, the Sorted L-One norm $J_{\lambda}$ is reduced to the $L_{\infty}$ norm. 

\begin{figure}[h!]
\centering
[$\lambda=(2,2,2)$]{\includegraphics[width=0.3\textwidth]{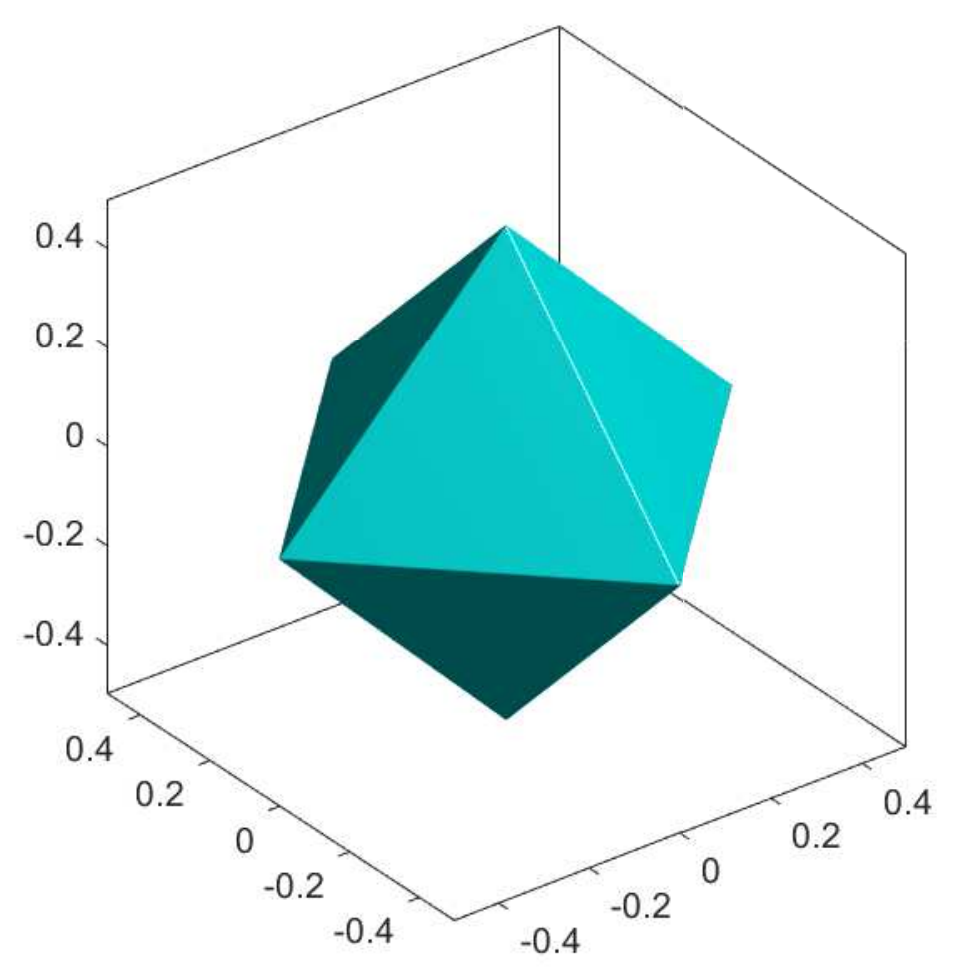}} 
[$\lambda=(2,0,0)$]{\includegraphics[width=0.3\textwidth]{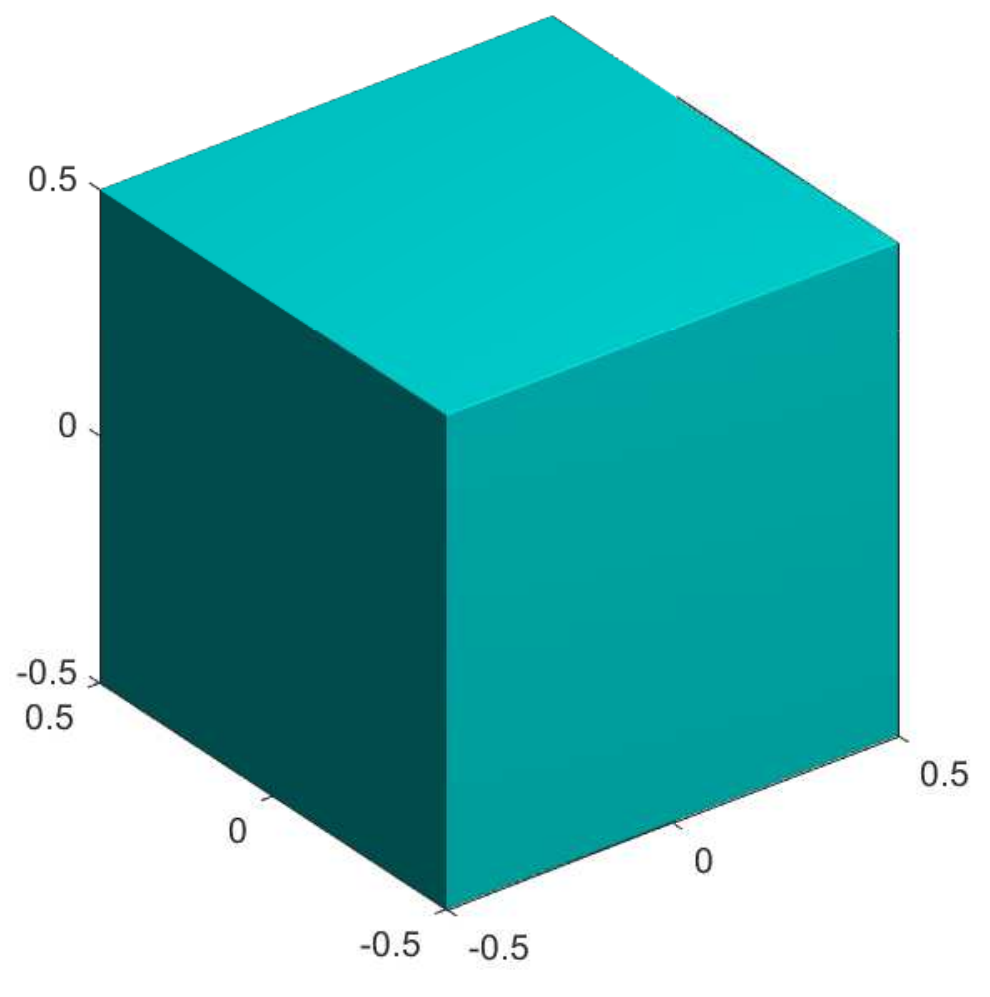}}
[$\lambda=(3,2,1)$]{\includegraphics[width=0.3\textwidth]{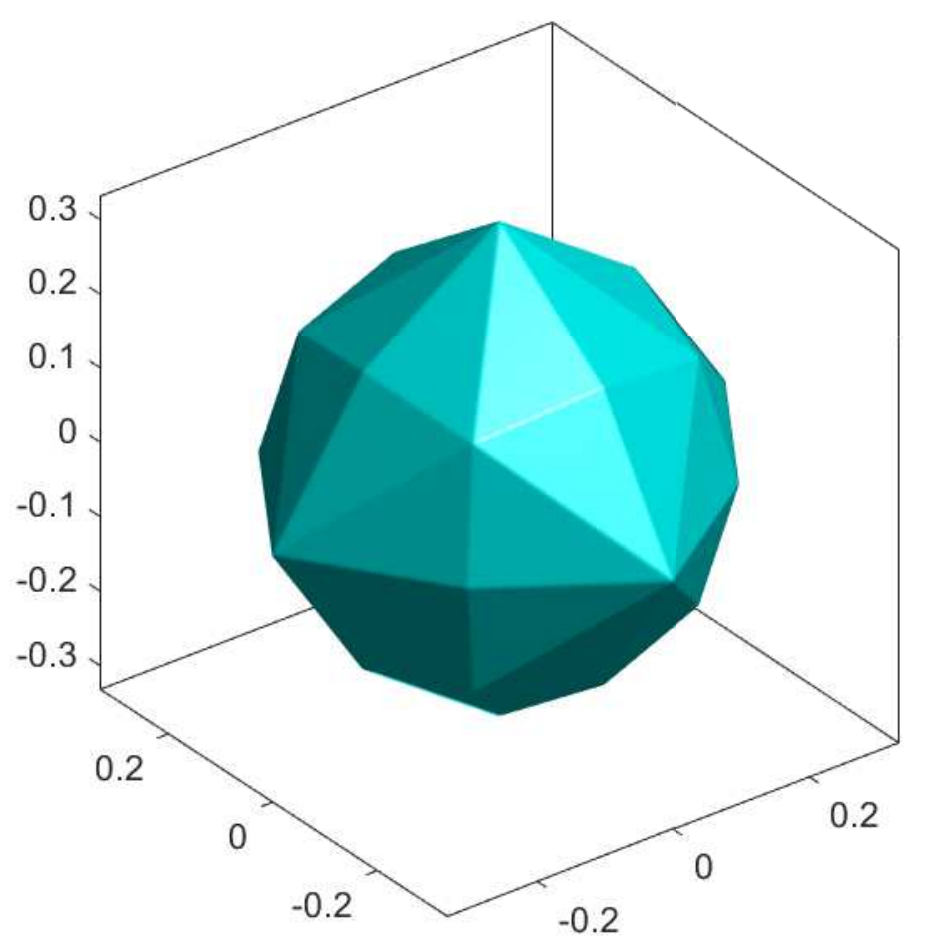}}
\caption{Shapes of different SLOPE spheres}
\label{balls}
\end{figure}

Figure \ref{balls} illustrates different shapes of the unit spheres corresponding to different versions of the Sorted L-One Norm. Since the solutions of SLOPE tend to occur on the edges of respective spheres, Figure 
\ref{balls} demonstrates the large flexibility of SLOPE with respect to the dimensionality reduction.
In the case when $\lambda_1=\ldots=\lambda_p$, SLOPE reduces dimensionality by shrinking the coefficients to zero. In the case when $\lambda_1>\lambda_2=\ldots=\lambda_p=0$, the reduction of dimensionality is performed by shrinking the coefficients towards each other (since the edges of the $l_{\infty}$ sphere correspond to vectors $b$ such that at least two coefficients are equal to each other). In the case when the sequence of thresholding parameters is monotonically decreasing, SLOPE reduces the dimensionality both ways: it shrinks them towards zero and towards each other. Thus it returns sparse and stable estimators, which have recently been proved to achieve minimax rate of convergence in the context of sparse high dimensional regression and logistic regression \citep{su2016, Bellec1, Felix}.

From the perspective of model selection it has been proved in \cite{SLOPE,SLOPE2} that SLOPE with the vector of tuning parameters $\lambda^{BH}$ (\ref{BH}) controls FDR at level $q$ under the orthogonal design. This is no longer true if the inner products between columns of the design matrix are different from zero, which almost always occurs if the predictors are random variables. Similar problems with the control of the number of False Discoveries occur for LASSO. Specifically, in \cite{SLOPE} it is shown that in the case of the Gaussian design with independent predictors, LASSO with a fixed parameter $\lambda$ can control FDR only if the true parameter vector is sufficiently sparse. The natural question is whether there exists a bound on the sparsity under which SLOPE can control FDR ? In this article we address this question and report a theoretical result on the asymptotic control of FDR by SLOPE under the Gaussian design. Our main  theoretical result states that by multiplying  sequence $\lambda^{BH}$  by a constant larger than 1, one can achieve the full asymptotic power and FDR converging to 0 if the number $k(n)$ of nonzero elements in the true vector of the regression coefficients satisfies $k=o\left(\sqrt{\frac{n}{\log p}}\right)$ and the values of these non-zero elements are sufficiently large. We also report results of a simulation study which suggest that the assumption on the signal sparsity is necessary when  using $\lambda^{BH}$ sequence but is unnecessarily strong when using the heuristic adjustment of this sequence, proposed in \cite{SLOPE2}. Simulations also suggest that the asymptotic FDR control is guaranteed independently of the magnitude of the non-zero regression coefficients.


\section{Asymptotic Properties of SLOPE}
\label{sec:FDR_charact}

\subsection{False Discovery Rate and Power}

Let us consider the multiple regression model (\ref{model}) 
and let $\hat b$ be some sparse estimator of $b^0$. The numbers of false, true and all rejections, and the number of non-zero elements in $b^0$ (respectively: $V$, $TR$, $R$, $k$), are defined as follows 
\begin{equation}
V = \#\{j: b_j^0=0\;\mbox{and}\;\hat b_j \neq 0\}, \ \ \ \ \ \ \ R = \#\{j:\hat b_j \neq 0\}\;\;,
\end{equation}
\begin{equation}
TR = \#\{j: b_j^0 \neq 0\;\mbox{and}\;\hat b_j \neq 0\}, \ \ \ \ \ \ \  k = \#\{j: b^0_j\neq 0\}\;\;.
\end{equation}
The False Discovery Rate is defined as:
\begin{equation}
FDR := \mathbb{E}\left(\frac{V}{R \vee 1}\right) 
\label{eq:FDR}
\end{equation}
and the Power  as:
\begin{equation}
\Pi := \frac{\mathbb{E} \left( TR\right)}{k} 
\end{equation}
\subsection{Asymptotic FDR and Power}

We formulate our asymptotic results under the setup where $n$ and $p$ diverge to infinity and $p$ can grow faster than $n$.  Similarly as in the case of the support recovery results for LASSO, we need to impose a constraint on the sparsity of $b^0$, which is measured by the number of truly important predictors $k = \#\{i: b^0_i \neq 0\}$. Thus we  consider the sequence of linear models of the form (\ref{model}) indexed by $n$ and with their "dimensions" characterized by the triplets: $(n,p_n,k_n)$.  For the sake of clarity, further in the text we skip the subscripts by $p$ and $k$.

The main result of this article is as follows:
\begin{theorem}\label{main}
Consider the linear model of the form (\ref{model}) and assume that all elements of the design matrix $X \in \mathbb{R}^{n \times p}$ are i.i.d. random variables from the normal $N(0,1/n)$ distribution.  Moreover, suppose there  exists $\delta >0$ such that
\begin{equation}
\min_{b^0_j \neq 0} \vert b^0_j \vert \geqslant 2\sigma (1+\delta)\sqrt{2 \log p}
\label{ampl}
\end{equation}
and suppose
\begin{equation}\label{ass1}
p\rightarrow \infty,\;\; \frac{k}{p} \rightarrow 0\;\;\mbox{and}\;\;\frac{k^2 \log p}{n} \rightarrow 0.\;\;
\end{equation} 
Then for any $q \in (0,1)$, the SLOPE procedure with the sequence of tuning parameters 
\begin{equation}\label{lBH}
\lambda_i(q,\delta) = \sigma (1+\delta) \Phi^{-1}\left(1- \frac{qi}{2p}\right)
\end{equation} 
has the following properties:\[
FDR \rightarrow 0
, \;\;\;\;\;\;\; \Pi \rightarrow 1.\]
\end{theorem}

\begin{proof}
The proof of Theorem \ref{main} makes extensive use of the results on the asymptotic properties of SLOPE reported in \cite{su2016}. The roadmap of this proof is provided in Section \ref{sec:RoadMap}, while the proof details can be found in the Appendix and in supplementary materials. 
\end{proof}

\begin{remark}[The assumption on the design matrix]
 The assumption that the elements of $X$ are i.i.d. random variables from the normal $N(0,1/n)$ distribution is technical. We expect that the results can be generalized to the case where the predictors are independent, sub-gaussian random variables. The assumption that the variance is equal to $1/n$ can be satisfied by an appropriate scaling of such a design matrix. 
As compared to the classical standardization towards unit variance, our scaling allows for the control of FDR with the sequence of the tuning parameters $\lambda$, which does not depend on the number of observations $n$. If the data are standardized such that $X_{ij}\sim N(0,1)$, then Theorem \ref{main} holds when the sequence of tuning parameters (\ref{lBH}) and  the lower bound on the signal magnitude (\ref{ampl}) are both multiplied by $n^{-0.5}$. 
\end{remark}

\begin{remark}[The assumption on the signal strength]
Our assumption on the signal strength is not very restrictive. When using the classical standardization of the explanatory variables (i.e. assuming that $X_{ij}\sim N(0,1)$), this assumption allows for the magnitude of the signal to converge to zero at such a rate that
$$\min_{b^0_j \neq 0} \vert b^0_j \vert \geqslant 2\sigma (1+\delta)\sqrt{\frac{2 \log p}{n}}\;\;.$$ 
This assumption is needed to obtain the power that converges to 1.  
The proof of Theorem \ref{main} implies that this assumption is not needed for the asymptotic FDR control  if $k$ is bounded by some constant. Moreover, our simulations suggest that the asymptotic FDR control holds independently of the signal strength if only $k$ satisfies the assumption (\ref{ass1}). The proof of this conjecture would require a substantial refinement of the proof techniques and remains an interesting topic for  future work.
\end{remark}

\subsection{Simulations}
In this section we present results of  the simulation study.
The data are generated according to the linear model:
\[Y = Xb^0 +\epsilon\;\;,\]
where elements of the design matrix $X$ are i.i.d. random variables from the normal $N(0,1/n)$ distribution and  $\epsilon$ is independent of $X$ and comes from the standard multivariate normal distribution $N(0,I)$. The parameter vector $b^0$ has $k$ non-zero elements and $p-k$ zeroes.

We present the comparison of the three methods - two versions of SLOPE and LASSO: 
\begin{enumerate}
\item SLOPE with the sequence of the tuning parameters  provided by the formula (\ref{lBH}), denoted by ''SLOPE".
\item SLOPE with the sequence:
\[
\lambda_i(q) = \Bigg \{\begin{array}{ccc}
\sigma \Phi^{-1}(1-q/2p) & if & i=1\\
\min\left(\lambda_{i-1}, \sigma\Phi^{-1}(1-qi/2p)\sqrt{1+\frac{\sum_{j<i}\lambda_j^2}{n-i-2}}\right) & if & i>1\;\;.
\end{array}
\]
This sequence was proposed in \cite{SLOPE2} as a heuristic correction which takes into account
the influence of the cross products between the columns of the design matrix $X$. We refer to this procedure as \textit{heuristic SLOPE} ("SLOPE\_heur").
\item LASSO with the tuning parameter:
\begin{equation}\label{lam}
\lambda= \sigma (1+\delta) \Phi^{-1}\left(1- \frac{q}{2p}\right)\;\;.
\end{equation}
\end{enumerate}
The tuning parameter for LASSO is equal to the first element of the tuning sequence of SLOPE, so FDR of LASSO and SLOPE are approximately equal when $k=0$.

\begin{figure}[h!]

        \centering
                \includegraphics[width=1\textwidth]{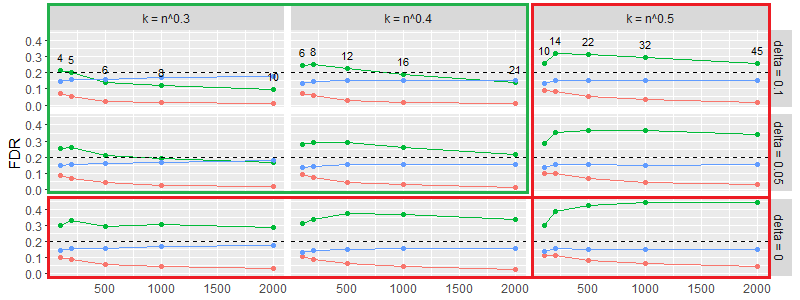} 
                \includegraphics[width=1\textwidth]{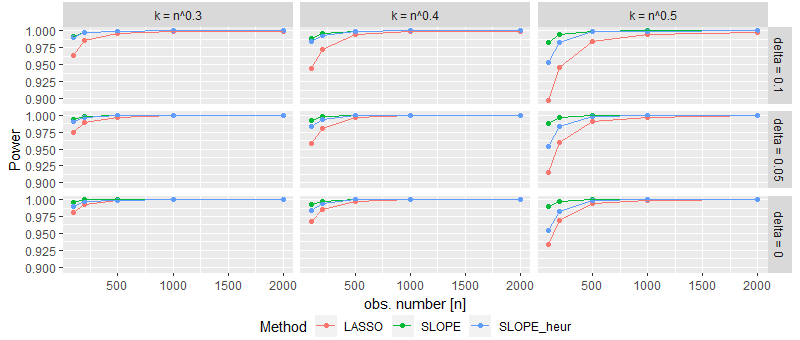} 
															\caption{FDR and Power of different procedures as  functions of $n$  and the parameters $\delta$ (see  (\ref{lBH}) and (\ref{lam})) and $\alpha$,  for $p=0.05n^{1.5}$, $k=round(n^{\alpha})$, $b^0_1=\ldots=b^0_k=2(1+\delta)\sqrt{2 \log p}$ and $q=0.2$. The results presented in a green rectangle correspond to the values of parameters which meet  assumptions of Theorem \ref{main}. 
The numbers above green lines correspond to actual values of the parameter $k$. Each point was obtained by averaging false and true positive proportions over at least 500 independent replicates.}
\label{FDR_str}															
 \end{figure}

Figure \ref{FDR_str} presents FDR and  Power  of different procedures.
First, let us concentrate on the behavior of SLOPE when the sequence of tuning parameters is defined as in Theorem \ref{main} (green line in each sub-plot). The green rectangle contains plots where the sequences of tuning parameters and the signal sparsity $k(n)$ satisfy the assumptions of Theorem \ref{main}. It is noticeable that in this area  FDR of SLOPE slowly converges to zero and the Power converges to 1. Moreover, FDR is close to or below the nominal level $q=0.2$ for the whole range of considered values of $n$. It is also clear that larger values of $\delta$ lead  to the more conservative versions of SLOPE. 

In the red area the assumptions of  Theorem \ref{main} are violated.
Here we can see that when $\delta>0$ and $\alpha=0.5$, FDR is still a decreasing function of $n$ but the rate of this decrease is slow and  
FDR is substantially above level $q=0.2$ even for $n=2000$. In the case when $\delta=0$ (i.e. when the original $\lambda^{BH}$ sequence is used), FDR stabilizes at the value which exceeds the nominal level. 

Let us now turn our attention to other methods. We can observe that LASSO is the most conservative procedure and that, as expected, its FDR converges to 0 when $k$ increases. Since the simulated signals are strong, this does not lead to a substantial decrease of Power as compared to SLOPE. Interestingly, SLOPE with a heuristic choice of tuning parameters seems to provide a stable FDR control over the whole range of considered parameter values. 
This suggests that the upper bound on $k$ provided in assumption (\ref{ass1}) could be relaxed when working with this heuristic sequence. The proof of this claim remains an interesting topic for further research. 


\begin{figure}[h!]

        \centering
                \includegraphics[width=1\textwidth]{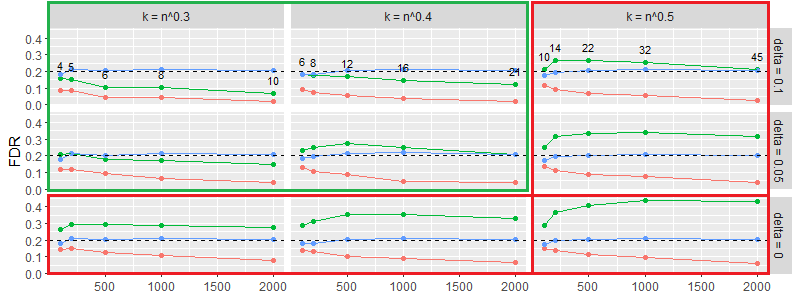} 
                \includegraphics[width=1\textwidth]{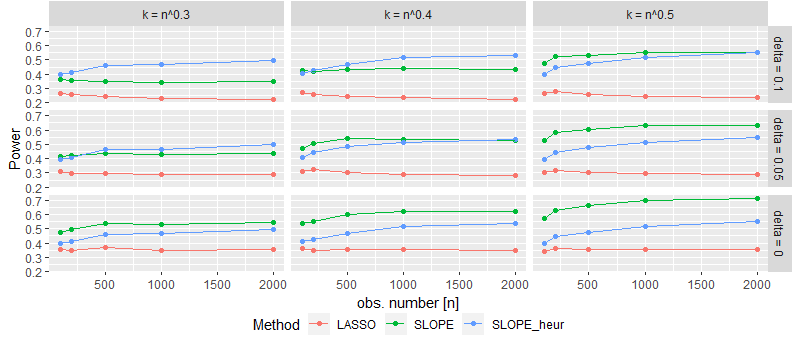} 
															\caption{FDR and Power of different procedures as  functions of $n$  and the parameters $\delta$ (see  (\ref{lBH}) and (\ref{lam})) and $\alpha$,  for $p=0.05n^{1.5}$, $k=round(n^{\alpha})$, $b^0_1=\ldots=b^0_k=0.9\sqrt{2 \log p}$ and $q=0.2$. The results presented in a green rectangle correspond to the values of parameters which meet  assumptions of Theorem \ref{main}, except the condition on signal strength. 
The numbers above green lines correspond to actual values of the parameter $k$. Each point was obtained by averaging false and true positive proportions over at least 500 independent replicates}
\label{FDR_small}
 \end{figure}
Figure \ref{FDR_small} presents simulations for the case when $b^0_1=\ldots=b^0_k=0.9\sqrt{2 \log p}$, i.e. when the signal magnitude does not satisfy the assumption (\ref{ampl}). Here FDR of SLOPE behaves similarly as in the case of strong signals. These results suggest that the assumption on the signal strength might not be necessary in the context of FDR control. Figure \ref{FDR_small} also illustrates a strikingly good control of FDR by the SLOPE with the heuristically adjusted sequence of tuning parameters.
LASSO is substantially more conservative than both versions of SLOPE. Its FDR converges to zero, which in the case of such moderate signals leads to a substantial decrease of Power as compared to SLOPE.

\section{Roadmap of the Proof}
\label{sec:RoadMap}

In the first part of this section we characterize the support of the SLOPE estimator.
The proofs of the Theorems presented in this part rely only on differentiability and convexity of the loss function. Therefore we decided to present them in a general form, will be useful for a further work on extensions of SLOPE for Generalized Linear Models or Gaussian Graphical Models.

\subsection{Support of the SLOPE estimator under the general loss function}
Let us consider the following generalization of SLOPE:
\begin{equation}\label{gSLOPE}
\hat b=argmin_{b} \left\{l(b)+\sum_{i=1}^p\lambda_i\big|b\big|_{(i)}\right\},
			\end{equation}
where $l(b)$ is a convex and a differentiable loss function (e.g. $0.5\Vert Y-X'b\Vert^2_2$ for the multiple linear regression).
Let $R$ denote the number of non-zero elements of $\hat b$. 


The following Theorems \ref{Lem2} and \ref{Lem1} characterize events $\{\hat{b}_i \neq 0\}$ and\\ $\{R=r\}$ by using the gradient $U(b)$ of the negative loss function  $-l(b)$;
\begin{equation}
U(b)=(U_1(b),...,U_p(b))'= -\triangledown (l(b))\;\;.
\label{score_v}
\end{equation}
Additionally, for $a >0$ we define the vector $T(a)$ as 
\[
T(a) = U(\hat{b})+ a\hat{b}\;\;.
\]
Thus, $T_i(a)=U_i(\hat{b})$ if $\hat{b}_i=0$. Also, the additional term $a\hat{b}_i$ has the same sign as  $U_i(\hat{b})$, so $|T_i(a)|>|U_i(\hat b)|$ if $\hat{b}_i\neq 0$. 
By calculating the subgradient of the objective function of the LASSO estimator, it is possible to check that LASSO selects these variables for which the respective coordinates of $|T(a)|$ exceed the value of the tuning parameter $\lambda$. In \cite{SLOPE2} the support of $\hat b$ for SLOPE under the orthogonal design is provided. It is shown that, similarly as in the case of the Benjamini-Hochberg correction for multiple testing, it is not sufficient to compare the ordered coordinates of $|T(a)|$ to the respective values of the decaying sequence of tuning parameters. It can happen that while performing this simple operation one could eliminate regressors with the value of $|T(a)|$ larger than for some of the regressors retained in the model. The SLOPE estimator preserves the ordering of $|T(a)|$. Thus, identification of the SLOPE support is more involved and  requires introduction of the following sets $H_r$: 
for $r\in\{1,\ldots,p\}$ we define 
\begin{equation}\label{hr}
H_r =  \Big\{
w \in \mathbb{R}^p: \forall_{j \leqslant r} \sum_{i=j}^r \lambda_i < \sum_{i=j}^r \vert w \vert_{(i)} \;\;\; and \;\;\; \forall_{j \geqslant r+1} \sum_{i=r+1}^j\lambda_i \geqslant \sum_{i=r+1}^j \vert w\vert_{(i)}
\Big\}.
\end{equation}

\begin{theorem}
\label{Lem2}
Consider the optimization problem $\ref{gSLOPE}$ with an arbitrary sequence $\lambda_1\geqslant \lambda_2 \geqslant ... \geqslant \lambda_p \geqslant 0$. Assume that $l(b)$ is a convex and a differentiable function. Then for any $a>0$, 
$$R = r \Longleftrightarrow T(a) \in H_r\;\;.$$
\end{theorem}

\begin{theorem}\label{Lem1}
Consider the optimization problem $\ref{gSLOPE}$ with an arbitrary sequence $\lambda_1\geqslant \lambda_2 \geqslant ... \geqslant \lambda_p \geqslant 0$. Assume that $l(b)$ is a convex and a differentiable function and $R = r$. Then for any $a>0$ it holds:
\[
\left( \hat{b}_i \neq 0 \right)
\Leftrightarrow \left(\vert T_i(a) \vert >\lambda_r\right) 
\]
and
\[
\left( \hat{b}_i \neq 0 \right)
\Rightarrow \left(\vert U_i(\hat b) \vert \geqslant\lambda_r\right) 
\]

Moreover if we assume that $\lambda_1>\lambda_2>...>\lambda_p\geqslant 0$, then:
\[
\left( \hat{b}_i \neq 0 \right)
\Leftarrow 
\left( \vert U_i(\hat b) \vert \geqslant \lambda_r \right)
\]
\end{theorem}
The proofs of Theorems \ref{Lem2} and \ref{Lem1} are provided in the supplementary materials.

\subsection{FDR of SLOPE for the general loss function}

\begin{corollary}\label{FDRscore}
Consider the optimization problem $\ref{gSLOPE}$ with an arbitrary sequence $\lambda_1\geqslant \lambda_2 \geqslant ... \geqslant \lambda_p \geqslant 0$. Assume that $l(b)$ is a convex and a differentiable function. Then for any $a>0$, FDR of SLOPE is equal to:
\begin{equation}
FDR = \sum_{r=1}^{p}\frac{1}{r}\sum_{i \in S^{c}}P\left(|T_i(a)|>\lambda_r,T(a) \in H_r \right)
\label{eq:FDR_T(a)}
\end{equation}
\end{corollary}
\begin{proof}

Let us denote the support of the true parameter vector $b^0$ by:
\begin{equation}
S = Supp(b^0)
\end{equation}
and a set that is  the complement of $S$ in $\{1,...,p\}$ by:
\begin{equation}
S^c = \{1,...,p\}\setminus S
\label{eq:comp_S}
\end{equation}
Directly from the definition we obtain:
\begin{eqnarray}\label{FDR_basic}
FDR& =& \mathbb{E}\left(\frac{V}{R \vee 1}\right) = \sum_{r=1}^{p}\frac{1}{r}\mathbb{E}(V\mathbbm{1}(R=r))=\sum_{r=1}^{p}\frac{1}{r}\mathbb{E}\left(\sum_{i \in S^{c}}\mathbbm{1}(\hat{b}_i \neq 0) \mathbbm{1}(R=r)\right) \nonumber \\ 
&=&\sum_{r=1}^{p}\frac{1}{r}\sum_{i \in S^{c}}\mathbb{E}(\mathbbm{1}(\hat{b}_i \neq 0) \mathbbm{1}(R=r)) = 
\sum_{r=1}^{p}\frac{1}{r}\sum_{i \in S^{c}}P(\hat{b}_i \neq 0,R=r)
\end{eqnarray}

and Corollary \ref{FDRscore} is a direct consequence of Theorems \ref{Lem2} and \ref{Lem1}. 
\end{proof}

\subsection{Proof of Theorem \ref{main}}
We now focus on the multiple regression model (\ref{model}).  
Elementary calculations show that in this case the vector $U$ (for def. see (\ref{score_v})) takes the form
\[
U(b) = X'(Y-Xb) = X'\epsilon + X'X(b^0-b) 
\]
Let us denote for simplicity 
\[
T = T(1) = \hat{b}+U(\hat{b}) = X'\epsilon + b^0 + (X'X - \mathbb{I})(b^0-\hat{b})
\]
and introduce the following notation for the components of $T$:
\begin{equation}
M = X'\epsilon + b^0
\end{equation}
and
\begin{equation}
\Gamma = (\mathbb{I}-X'X)(\hat{b}-b^0)\;\;.
\end{equation}
Naturally, $T = M + \Gamma$.\\
Due to (\ref{eq:FDR_T(a)}), we can express FDR for linear regression in the following way:
\begin{equation}
FDR = \sum_{r=1}^{p}\frac{1}{r}\sum_{i \in S^{c}}P(T \in H_r,\vert T_i \vert > \lambda_r)\;\;.
\label{FDR_pom1}
\end{equation}

Deeper analysis shows that under the assumptions of Theorem \ref{main}, the FDR expression (\ref{eq:FDR_T(a)}) can be simplified.  Corollary \ref{resolvent}, stated below, follows directly from Lemma 4.4 in \citep{su2016} (see the supplementary materials) and shows that, with a large probability, only the first elements of the summation over $r$  are different from zero. Furthermore,  the following Lemma \ref{Gamma_cond} shows that elements of the vector $\Gamma$ are sufficiently small, so we can focus on the properties of the vector $M$.

\begin{definition}[Resolvent set, \cite{su2016}] 
Fix $S = supp(b^0)$ of cardinality k, and an integer $\tilde k^*$
obeying $k < \tilde k^* < p$. The set $\tilde S^* = \tilde S^*(S, \tilde k^*)$ is said to be a resolvent set if it is the union of S and
the $\tilde k^*-k$ indexes with the largest values of 
$\vert X_i' \epsilon \vert$ among all $i \in \{1,...,p\} \setminus S$.
\end{definition}

Let us introduce the following notation on a sequence of events when the union of supports of  $b^0$ and $\hat b$ is contained in $\tilde S^*$
\begin{equation}
Q_1(n,\tilde k^*) = \{Supp(b^0)\cup Supp(\hat{b})\subseteq \tilde S^*\}
\label{Q_1(n)}
\end{equation}

\begin{corollary}
Suppose the assumptions of Theorem \ref{main} hold.
Then there exists a deterministic sequence $k^*$ such that $k^*/p \rightarrow 0$, $((k^*)^2 \log p)/n \rightarrow 0$  and:
\begin{equation}
P(Q_1(n, k^*)) \rightarrow 1.
\label{b_in_S*}
\end{equation}
\label{resolvent}
\end{corollary}
Corollary \ref{resolvent} follows directly from Lemma 4.4 in \cite{su2016} (see Lemma S.2.2 in the supplementary materials and the discussion below). From now on $k^*$ will denote the sequence satisfying  Corollary \ref{resolvent}.

\begin{lemma}
\label{Gamma_cond}
Let us denote by $Q_2(n, \tilde{\gamma}(n))$ a sequence of events when the $l_{\infty}$ norm of vector $\Gamma$ is smaller than $\tilde \gamma(n)$:
\begin{equation}
 Q_2(n,\tilde \gamma (n)) = \{\max_i\vert \Gamma_i \vert \leqslant \tilde\gamma(n) \}\;\;.
 \end{equation} 
If the assumptions of Theorem \ref{main} hold then there exists a constant $C_q$, dependent only on $q$, such that the 
sequence $\gamma(n) = C_q \sqrt{\frac{(k^*)^2 \log p}{n}} \lambda^{BH}_{k^*}$ satisfies:
\begin{equation}
P\left(Q_2(n,\gamma(n)) \right) \rightarrow 1\;\;.
\label{Gamma_0}
\end{equation}
\end{lemma}
The proof of Lemma \ref{Gamma_cond} is provided in the Appendix.

Let us denote by $Q_3(n,u)$ a sequence of events such that the  $l_2$ norm of the vector $\epsilon$ divided by $\sigma \sqrt{n}$ is smaller than $1 + 1/u$:
\begin{equation}
Q_3(n,u) = \Big\{ \frac{\Vert \epsilon \Vert_2}{\sigma\sqrt{n}} \leqslant 1+1/u \Big\}
\label{Q3}
\end{equation} 

The following Corollary \ref{epsilon} is a consequence of the well known results on the concentration of the Gaussian measure (see Theorem S.2.4 in the supplementary materials).

\begin{corollary}
Let $k^*=k^*(n)$ be the sequence satisfying Corollary \ref{resolvent}. Then
\begin{equation}
P\left(Q_3(n,k^*) \right)\geqslant 1 - e^{- \frac{n}{2(k^*)^2}} \rightarrow 1
\label{Bor}
\end{equation}
\label{epsilon}
\end{corollary}

From now on, for simplicity, we shall denote by  $Q_1, Q_2$ and $Q_3$ the sequences $Q_1(n,k^*), Q_2(n,\gamma)$ and $Q_3(n,k^*)$ respectively. Moreover, let us introduce the following notation on the intersection of $Q_1, Q_2$ and $Q_3$:
\begin{equation}
Q = Q(n) = Q_1(n,k^*) \cap Q_2(n,\gamma) \cap Q_3(n,k^*)
\label{Q}
\end{equation} 

By using an event $Q$, we can bound FDR in the following way:
 
\[
FDR = \mathbb{E}\left(\frac{V}{R\vee 1}\mathbbm{1}( Q^c)\right) + \mathbb{E}\left(\frac{V}{R\vee 1}\mathbbm{1}( Q)\right)   
 \leqslant \mathbb{E}\left(\mathbbm{1}( Q^c)\right) + \mathbb{E}\left(\frac{V}{R\vee 1}\mathbbm{1}( Q)\right) = \] 
\begin{equation}
= P(Q^c) + 
\sum_{r=1}^{p}\frac{1}{r}\sum_{i \in S^{c}}
P(R=r,\hat{b}_i \neq 0,Q)
\label{FDR_pom11}
\end{equation}
The first equality follows from the fact that $\mathbbm{1}( Q) + \mathbbm{1}(Q^c) = 1$. The inequality uses the fact that $\frac{V}{R\vee 1} \leqslant 1$. The second equality is a consequence of the formula (\ref{FDR_basic}) applied to the second term.
Naturally, due to conditions (\ref{b_in_S*}), (\ref{Gamma_0}) and (\ref{Bor}), we obtain  $P(Q^c)\rightarrow 0$. Therefore, we can focus on the properties of the second term in (\ref{FDR_pom11}).  Notice that $Q_1$ implies that $ R\leqslant k^*$ ($supp(\hat b) \subset S^*$), therefore we can limit the summation over $r$ to the first $k^*$ elements:
\begin{equation}
\sum_{r=1}^{p}\frac{1}{r}\sum_{i \in S^{c}}
P(R=r,\hat{b}_i \neq 0,Q)= \sum_{r=1}^{k^*}\frac{1}{r}\sum_{i \in S^{c}}P(R=r,\hat{b}_i \neq 0,Q)\;\;.
\label{FDR_pomoc1}
\end{equation}
Furthermore, according to Theorems \ref{Lem2} and \ref{Lem1}:
\begin{equation}
\sum_{r=1}^{k^*}\frac{1}{r}\sum_{i \in S^{c}}
P(R=r,\hat{b}_i \neq 0,Q)= \sum_{r=1}^{k^*}\frac{1}{r}\sum_{i \in S^{c}}P(T \in H_r,\vert T_i \vert> \lambda_r,Q)\;\;.
\label{FDR_pomoc2}
\end{equation}

We now introduce the useful notation:
\begin{itemize}
\item a vector $M^{(i)} = (M^{(i)}_1,...,M^{(i)}_p)'$ is the following modification of $M$ 
\[
M^{(i)}_j:=
\Bigg\lbrace\begin{array}{l}
M_j  \;\;\; i\neq j \\
\infty   \;\;\; i=j\end{array}
\]
\item a set $H_r^{\gamma}$, which is a generalization of the set $H_r$ (\ref{hr}),
\[H_r^{\gamma} = \{
w \in \mathbb{R}^p: \forall_{j \leqslant r} \sum_{i=j}^r (\lambda_i- \gamma) < \sum_{i=j}^r \vert w \vert_{(i)} \;\;\; and \;\;\; \forall_{j \geqslant r+1} \sum_{i=r+1}^j(\lambda_i+ \gamma) \geqslant \sum_{i=r+1}^j \vert w\vert_{(i)}
\}\;\;.\]
\end{itemize}
Lemma \ref{T_to_M} (see the Appendix for the proof) allows for the replacement of an event $\{T \in H_r,\vert T_i \vert> \lambda_r,Q_2\}$ by an event which depends only on $M^{(i)}$.
\begin{lemma}
\label{T_to_M}
If $T \in H_r$, $\vert T_i \vert> \lambda_r$ and $Q_2$ occurs, then $M^{(i)} \in H_r^{\gamma}$.
\end{lemma}
Lemma \ref{T_to_M}, together with a fact that under $Q_2$ $\vert T_i \vert> \lambda_r$,
allows for the conclusion that $\vert M_i \vert> \lambda_r - \gamma$ 
and therefore
\begin{equation}
\sum_{r=1}^{k^*}\frac{1}{r}\sum_{i \in S^{c}}P(T \in H_r,\vert T_i \vert> \lambda_r,Q) \leqslant  
\sum_{r=1}^{k^*}\frac{1}{r}\sum_{i \in S^{c}}P\left(M^{(i)} \in H_r^{\gamma},\vert M_i \vert> \lambda_r-\gamma,Q_3\right) 
\label{FDR_M}
\end{equation}
The following Lemma \ref{Conv} (see the Appendix for the proof ) provides asymptotic behavior of the right-hand side of (\ref{FDR_M}):

\begin{lemma}
\label{Conv}
Under the assumptions of Theorem \ref{main}, it holds:
\begin{equation}
\sum_{r=1}^{k^*}\frac{1}{r}\sum_{i \in S^{c}}P\left(M^{(i)} \in H_r^{\gamma},\vert M_i \vert> \lambda_r-\gamma,Q_3\right) 
\rightarrow 0
\end{equation}
\end{lemma}

The proof of Lemma \ref{Conv} is based on several properties. First,
\[
P\left(M^{(i)} \in H_r^{\gamma},\vert M_i \vert> \lambda_r-\gamma,Q_3\right)\]
can be well approximated by
\[P\left(M^{(i)} \in H_r^{\gamma},Q_3\right) P\left(\vert M_i \vert> \lambda_r-\gamma,Q_3\right) 
\]
This approximation is a consequence of the fact that conditionally on $\epsilon$, $M^{(i)}$ and $M_i$ are 
independent. Second, for $i \in S^c$:
\[
P\left(\vert M_i \vert> \lambda_r-\gamma,Q_3\right)\]
can be well approximated by 
\[ 2(1-\Phi((1+\delta)\lambda_r^{BH})) \approx \left(\frac{rq}{p} \right)^{(1+\delta)^2}\;\;,
\]
where $\Phi(\cdot)$ is the cumulative distribution function of the standard normal distribution. The first approximation is a consequence of the fact that for $i \in S^c$, $M_i= X_i\epsilon$. Thus, conditionally on $\epsilon$, $M_i$ has a normal distribution with the mean equal to 0 and the variance equal to $\frac{||\epsilon||^2}{n}$, which is close to $\sigma^2$ (see Corollary \ref{epsilon}). The second approximation relies on the well known formula
\begin{equation}\label{clas}
1-\Phi(c)=\frac{\phi(c)}{c}(1+o(c))\;\;,
\end{equation}
where $\phi(\cdot)$ is the density of the standard normal distribution and $o(c)$ converges to zero as $c$ diverges to infinity.

Lastly
\[
\sum_{r=1}^{k^*}P\left(M^{(i)} \in H_r^{\gamma}\right) \]
equals approximately to
\[ \sum_{r=1}^{k^*}P\left(M^{(i)} \in H_r\right) \leqslant 1
\]
This approximation is a consequence of the fact that $H_r^\gamma$ does not differ much from $H_r$ and the family of the sets $H_r$ with $r\in\{1,\ldots,p\}$ is disjoint. By applying the above approximations we obtain:
\[
\sum_{r=1}^{k^*}\frac{1}{r}\sum_{i \in S^{c}}P\left(M^{(i)} \in H_r^{\gamma},\vert M_i \vert> \lambda_r-\gamma,Q_3\right) 
\approx 
\sum_{i \in S^{c}}\sum_{r=1}^{k^*}\frac{1}{r}
\left(\frac{rq}{p} \right)^{(1+\delta)^2}P\left(M^{(i)} \in H_r^{\gamma}\right) \leqslant
\]
\[
\left(\frac{q}{p} \right)\left(\frac{k^*q}{p} \right)^{(1+\delta)^2-1}\underbrace{\sum_{i \in S^{c}}\sum_{r=1}^{k^*}
P\left(M^{(i)} \in H_r^{\gamma}\right)}_{\approx p-k} \approx 
\left(\frac{(p-k)q}{p} \right)\left(\frac{k^*q}{p} \right)^{(1+\delta)^2-1} \rightarrow 0
\]

 Lemma \ref{Conv}, together with the fact that under assumptions of Theorem \ref{main} $P(Q^c)\rightarrow 0$, provides $FDR \rightarrow 0$. One can notice that $\left(\frac{k^*q}{p} \right)^{(1+\delta)^2-1}$ is the factor responsible for the convergence of FDR to 0. It remains an open question if in the definition of the $\lambda$ sequence (\ref{lBH}) a constant $\delta$ can be replaced by a sequence converging to 0 at such a rate that the asymptotic FDR is exactly equal  to the nominal level $q$. The proof of this assertion would require a refinement of the bounds provided in \cite{su2016} and we leave this as a topic for future research.

Now, we will argue that under our assumptions the Power of SLOPE converges to 1 . Recall that $TR = \#\{j: b_j^0 \neq 0\;\mbox{and}\;\hat b_j \neq 0\}$ denotes the number of true rejections. Observe that
\[
\Pi = \frac{1}{k}\mathbb{E}\left(TR \right)=
 \frac{1}{k} \sum_{j = 1}^k jP(TR=j) \geqslant
P(TR = k) =  P\left(\bigcap_{i \in S}\{\hat b_i \neq 0 \}\right)
\]
Naturally $\Pi \leqslant 1$, therefore by showing that $P\left(\bigcap_{i \in S}\{\hat b_i \neq 0 \}\right) \rightarrow 1$ we obtain the thesis. 

\begin{lemma}
\label{Pi}
Under the assumptions of Theorem \ref{main}, it holds:
\begin{equation}
P\left(\bigcap_{i \in S}\{\hat b_i \neq 0 \}\right) \rightarrow 1.
\end{equation}
\end{lemma}
\begin{proof} The proof of Lemma \ref{Pi} is provided in the Appendix. It is based on the sequence of the following inequalities:
$$P\left(\bigcap_{i \in S}\{\hat b_i \neq 0 \}\right)\geq  P\left(\bigcap_{i \in S}\{|T_i|>\lambda_1\}\right)\geq P\left(\bigcap_{i \in S}\{|M_i|>\lambda_1+\gamma \}\right)$$
$$\geq P\left(\bigcap_{i \in S}\{ X_i' \epsilon> \lambda_1+\gamma-|b_i^0| \}\right)\geq  P\left(\bigcap_{i \in S}\{ X_i' \epsilon> -\sigma\left(1+\frac{\delta}{2}\right)\sqrt{2\log p}\right)\;\;.$$
The result follows by the conditional independence of $X_i'\epsilon$ (given $\epsilon$) and the classical approximation (\ref{clas}).

\end{proof}

\section{Discussion}

In this article we provide new asymptotic results on the model selection properties of SLOPE in the case when the elements of the design matrix $X$ come from the normal distribution. Specifically, we provide conditions on the sparsity and the magnitude of true signals such that FDR of SLOPE based on the sequence $\lambda^{BH}$, corresponding to the thresholds of the Benjamini-Hochberg correction for multiple testing, converges to 0 and the Power converges to 1. We believe these results can be extended to the sub-gaussian design matrices with independent columns, which is the topic of ongoing research. Additionally, the general results on the support of SLOPE  open the way for investigation of the properties of SLOPE under arbitrary convex and differentiable loss functions.

In simulations we compared SLOPE based on the sequence $\lambda^{BH}$ with SLOPE based on the heuristic sequence proposed in \cite{SLOPE2} and with LASSO with the tuning parameter $\lambda$ adjusted to the first value of the tuning sequence for SLOPE. When regressors are independent and the vector of true regression coefficients is relatively sparse then the comparison between SLOPE and LASSO bears similarity to the comparison between the Bonferroni and the Benjamini-Hochberg corrections for multiple testing. When $k=0$ both methods perform similarly and control FDR (which for $k=0$ is equal to the Family Wise Error Rate) close to the nominal level. When $k$ increases, FDR of LASSO converges to zero, which however comes at the prize of a substantial loss of Power for moderately strong signals.  Concerning the two versions of SLOPE, our simulations suggest that the heuristic sequence allows for a more accurate FDR control over a wide range of sparsity values.  We believe the techniques developed in this article form a good foundation for the analysis of the statistical properties of the heuristic version of SLOPE, which we consider an interesting topic for further research.

 Our assumptions on the independence of predictors and the sparsity of the vector of true regression coefficients are  restrictive, which is related to the well known problems with FDR control by  LASSO \citep{SLOPE,FDR_LASSO}. In the case of LASSO, these problems can be solved by using adaptive or reweighted LASSO \citep{Zou, Candes}, which allow for the consistent model selection under much weaker assumptions. In these modifications the values of the tuning parameters corresponding to predictors which are deemed important (based on the values of initial estimates) are reduced, which substantially reduces the bias due to shrinkage and improves model selection accuracy. Recently \cite{Wei} developed the Adaptive Bayesian version of SLOPE (ABSLOPE). According to the results of simulations presented in \citep{Wei}, ABSLOPE controls FDR under a much wider set of scenarios than the regular SLOPE, including examples with strongly correlated predictors. We believe our proof techniques can be extended to derive asymptotic FDR control for ABSLOPE, which we leave as an interesting topic for future research.

\section{Acknowledgment}

We would like to thank the Associate Editor and the Referees for many constructive comments. Also, we would like to thank Damian Brzyski for preparing Figure \ref{balls} and Artur Bogdan for helpful suggestions. The research was funded by 
the grant of the Polish National Center of Science Nr 2016/23/B/ST1/00454.

\section{Appendix}

\begin{proof}\textit{The proof of Lemma \ref{Gamma_cond}}\\

We prove a stronger condition: $P(Q_1 \cap Q_2) \rightarrow 1$.\\

Denote by $X_I,\hat{b}_I,b^0_I$  a submatrix (subvector) of $X ,\hat{b},b^0$ consisting of columns (vector elements) with indexes in a set $I$.

Observe that when $Q_1 = \{Supp(b^0)\cup Supp(\hat{b})\subseteq S^*\}$, we can express $\max_i\vert \Gamma_i \vert$ in the following way:
\[
\max_i\vert \Gamma_i \vert = \max (
[\max_{i}\vert [X_{{(S^*)}^c}'X_{S^*}(\hat{b}_{S^*}-b^0_{S^*})]_{i} \vert ,
\max_{i }\vert [(X_{S^*}'X_{S^*}-\mathbb{I})(\hat{b}_{S^*}-b^0_{S^*})]_{i} \vert )
\]
We show that both elements are bounded by $C_q \sqrt{\frac{(k^*)^2 \log p}{n}} \lambda^{BH}_{k^*}$ with the probability tending to 1.
The bound on the first component is a direct corollary from  Lemma A.12 proved in \citet{su2016}.
\begin{corollary}
Under the assumptions of Theorem \ref{main}, there exists a constant $C_q$ only depending on q such that:
\begin{equation}
\max_{i \in {(S^*)}^c}\vert [X_{{(S^*)}^c}'X_{S^*}(\hat{b}_{S^*}-b^0_{S^*})]_{i} \vert 
\leqslant
C_q \sqrt{\frac{k^* \log p}{n}} \lambda^{BH}_{k^* +1}
\label{eq:L3.2.Sec}
\end{equation}
with the probability tending to 1.
\end{corollary}
Lemma A.12 of \citet{su2016} and the proof of Corollary \ref{eq:L3.2.Sec} can be found in the supplementary materials (see Lemma S.2.5 and the discussion below).

It remains to be proven that the second component is also bounded by $C_q \sqrt{\frac{(k^*)^2 \log p}{n}} \lambda^{BH}_{k^*}$ with the probability tending to 1. To do this, we use Lemma A.11 proved in \citet{su2016} which provides bounds on the largest and the smallest singular values of $X_{S^*}$ (for details see Lemma S.2.6 in the supplementary materials).

Let $X_{S^*} = G\Sigma V'$
 be the Singular Value Decomposition (SVD) of the matrix $X_{S^*}$, where $G \in \mathbb{M}_{n \times n}$ is a unitary matrix, $\Sigma \in \mathbb{M}_{n \times k^*}$ is a diagonal matrix with non-negative real numbers on the diagonal and $V \in \mathbb{M}_{k^* \times k^*}$ is a unitary matrix. Moreover, let us denote by $\sigma_i, \sigma_{min}$ and $\sigma_{max}$ the i-th, the smallest and the largest singular value.
 
Assume that $u$ is an arbitrary unit vector. Due to the relation between the $l_{\infty}$ and the $l_2$ vector norms, and additionally between the $l_2$ vector norm and the $\Vert \cdot \Vert_2$ matrix norm, it holds:
\[
\max_{i}\vert [(X_{S^*}'X_{S^*}-\mathbb{I})u]_{i} \vert \leqslant \Vert (X_{S^*}'X_{S^*}-\mathbb{I})u \Vert_2
\leqslant \Vert X_{S^*}'X_{S^*}-\mathbb{I} \Vert_2 \;\;.
\]
Using the SVD of the matrix $X_{S^*}$ and the sub-multiplicity of the $\Vert \cdot \Vert_2$ matrix norm we obtain:
\[
=\Vert V(\Sigma'\Sigma-\mathbb{I})V' \Vert_2 \leqslant 
\Vert V \Vert_2 \Vert \Sigma'\Sigma-\mathbb{I} \Vert_2 \Vert V \Vert_2 = \Vert \Sigma'\Sigma-\mathbb{I} \Vert_2 = \sigma_{max}(\Sigma'\Sigma-\mathbb{I})
\]
By definition the maximal singular value of the matrix $A$ equals to the square root of the largest eigenvalue of the positive-semidefinite matrix  $A'A$. Therefore in our case it is obvious that:
\[
\sigma_{max}(\Sigma'\Sigma-\mathbb{I}) = \max_i(|\sigma_i^2 - 1|) = \max(|\sigma_{min}^2 - 1|,|\sigma_{max}^2 - 1|) \;\;,
\]
where $\sigma_i$ is the i-th singular value of $X_{S^*}$. Due to Lemma A.11 of \citet{su2016} (see Lemma S.2.6 in the supplementary materials) ,we obtain that for some constants $C_1$ and $C_2$
\[
\max(|\sigma_{min}^2 - 1|,|\sigma_{max}^2 - 1|) 
\leqslant C_1 \max(|\sigma_{min} - 1|,|\sigma_{max} - 1|) \leqslant C_2\sqrt{k^* \log(p/k^*)/n}
\]
and in consequence for an arbitrary unit vector $u$
\begin{equation}
\max_{i}\vert [(X_{S^*}'X_{S^*}-\mathbb{I})u]_{i} \vert \leqslant C_2\sqrt{k^* \log(p/k^*)/n}
\label{XX_I2}
\end{equation}
with the probability at least $1 - 2e^{-k^* \log(p/k^*)/2} - (\sqrt{2}ek^*/p)^{k^*}$.

On the other hand, due to Theorem 1.2 in \cite{su2016}, for any constant $\delta_1>0$  
\begin{equation}
P\left(\Vert \hat{b}-b^0 \Vert_2<(1+\delta_1)\sigma \sqrt{2  k \log\left(p/k\right)}\right) \rightarrow 1\;\;.
\label{b_bhat}
\end{equation}
Finally, observe that
\[
\max_{i }\vert [(X_{S^*}'X_{S^*}-\mathbb{I})(\hat{b}_{S^*}-b^0_{S^*})]_{i} \vert = 
\max_{i }\vert [(X_{S^*}'X_{S^*}-\mathbb{I})\frac{\hat{b}_{S^*}-b^0_{S^*}}{\Vert \hat{b}_{S^*}-b^0_{S^*} \Vert_2} ]_{i} \vert \Vert \hat{b}_{S^*}-b^0_{S^*} \Vert_2
\;\;.\]
Thus, the relations (\ref{XX_I2}) and (\ref{b_bhat}) imply that for certain constants $C_1$ and $C_2$:
\begin{equation}
\max_{i }\vert [(X_{S^*}'X_{S^*}-\mathbb{I})(\hat{b}_{S^*}-b^0_{S^*})]_{i} \vert  \leqslant
C_1\sqrt{\frac{k^*k \log(p/k)}{n}}\sqrt{\log(p/k^*)} \leqslant C_2\sqrt{\frac{(k^*)^2 \log(p)}{n}}\lambda^{BH}_{k^*}
\label{XX_I3}
\end{equation}
with the probability tending to 1.\\
The inequality (\ref{XX_I3}), together with (\ref{eq:L3.2.Sec}), provide the thesis of the Lemma.
\end{proof}

\begin{proof} \textit{The proof of Lemma \ref{T_to_M}}\\\\
In order to prove Lemma \ref{T_to_M}, we have to introduce the modification of a vector $T$ similar to that of vector M. 
Denote $T^{(i)} =(T^{(i)}_1,...,T^{(i)}_p)'$:
\[T^{(i)}_j:=
\Bigg\lbrace\begin{array}{l}
T_j  \;\;\; i\neq j \\
\infty \;\;\; i=j\end{array}
\]
In the first step we  show that:
\begin{proposition}
If $T \in H_r$ and $\vert T_i \vert> \lambda_r$ occurs, then $T^{(i)} \in H_r$
\end{proposition}

\begin{proof}
Let us recall the definition of a set $H_r$:
\[H_r = \{
w \in \mathbb{R}^p: \forall_{j \leqslant r} \sum_{l=j}^r \lambda_l < \sum_{l=j}^r \vert w \vert_{(l)} \;\;\; and \;\;\; \forall_{j \geqslant r+1} \sum_{l=r+1}^j\lambda_l \geqslant \sum_{l=r+1}^j \vert w\vert_{(l)}
\}
\]
On the one hand, we know that $\vert T_i \vert> \lambda_r$. On the other hand, $T \in H_{r}$ implies that $\vert T \vert_{(r+1)} \leqslant \lambda_{r+1}$ (the second condition in the definition of $H_r$ for $j = r+1$). These inequalities  imply together that $\vert T_i \vert \geqslant \vert T \vert_{(r)}$. Hence
we only have to show  that:
\begin{equation}
(\forall_{j \leqslant r} \sum_{l=j}^r \lambda_l < \sum_{l=j}^r \vert T \vert_{(l)})
\Rightarrow
(\forall_{j \leqslant r} \sum_{l=j}^r \lambda_l < \sum_{l=j}^r \vert T^{(i)} \vert_{(l)})
\label{im:1}
\end{equation}
To see that the above is true we have to consider the relations between the $r$  largest elements of vectors $\vert T \vert$ and $\vert T^{(i)} \vert$. Let us assume that $\vert T_i \vert = \vert T \vert_{(k)}$ for some $k \leqslant r$. By the definition of the vector $\vert T^{(i)} \vert$ we know that $\vert T^{(i)} \vert_{(1)} = \vert T^{(i)}_i \vert  = \infty$ and that the other ordered statistics of $\vert T^{(i)} \vert$ are related to the ordered statistics of $\vert T \vert$ in the following way:
\[
\forall_{2 \leqslant s \leqslant k} \;\;\;
\vert T^{(i)} \vert_{(s)} = \vert T \vert_{(s-1)} 
\]
\[
\forall_{k+1 \leqslant s \leqslant r} \;\;\;
\vert T^{(i)} \vert_{(s)} = \vert T \vert_{(s)} 
\]
In consequence we obtain that:
\[
\forall_{1 \leqslant s \leqslant r} \;\;\;
\vert T^{(i)} \vert_{(s)} \geqslant \vert T \vert_{(s)} 
\]
and this implies (\ref{im:1}).
\end{proof}
In the second step of the proof of Lemma \ref{T_to_M} we show that:
\[
T^{(i)} \in H_r \land Q_2 \Rightarrow M^{(i)} \in H^{\gamma}_r
\]

In order to prove this, we use the following Proposition:

\begin{proposition}
\label{prop2}
Let us assume we have three vectors $A,B,C \in \mathbb{R}^p$ and that $A = B+C$. Furthermore, let us assume that the vector $A$ is ordered ($\vert A_1 \vert \geqslant...\geqslant \vert A_p\vert$)
and define $d = \sup_i \vert C_i \vert$.
Under the above assumptions:
\[
 \Big\vert \vert A_i \vert - \vert B \vert_{(i)} \Big\vert \leqslant d
\]
for all $i = 1,...,p$.
\end{proposition}
\begin{proof}
 Let 
$$e = \sup_i \Big\vert \vert A_i \vert - \vert B_i \vert \Big\vert\;.$$
From the triangular inequality, we have:
\[
\vert C_i \vert = \vert A_i-B_i \vert\geqslant \Big\vert \vert A_i \vert - \vert B_i \vert \Big\vert \;\;\;\;\; \mbox{for}\;\;\; i = 1,2,...,p
\]
In consequence, $d \geqslant e$. Thus, to prove the Proposition \ref{prop2}
it is sufficient to prove that
\[\Big\vert \vert A_i \vert - \vert B \vert_{(i)} \Big\vert \leqslant e  \;\;\;\;\; \mbox{for}\;\;\; i = 1,2,...,p\;.
\]

For this aim, let $i_0$ be the index such that:
\[
\Big\vert \vert A_{i_0} \vert - \vert B \vert_{(i_0)} \Big\vert = \sup_i \Big\vert \vert A_i \vert - \vert B \vert_{(i)} \Big\vert = f\;.
\]
When $\vert A_{i_0} \vert = \vert B \vert_{(i_0)}$, then $f=0$ and the thesis is obtained immediately.\\
Let us consider the case when $\vert A_{i_0} \vert > \vert B \vert_{(i_0)}$.\\ 
If $\vert B_{i_0} \vert \leqslant \vert B \vert_{(i_0)}$, then:
\[
f = \Big\vert \vert A_{i_0} \vert - \vert B \vert_{(i_0)} \Big\vert = \vert A_{i_0} \vert - \vert B \vert_{(i_0)} \leqslant \vert A_{i_0} \vert - \vert B_{i_0} \vert\leqslant e
\]
and in consequence we obtain the thesis.\\
If $\vert B_{i_0} \vert \ > \vert B \vert_{(i_0)}$, then:
\[
\vert B_{i_0} \vert \in \mathbb{B} = \Big\{\vert B \vert_{(1)},\vert B \vert_{(2)},...,\vert B \vert_{(i_0-1)} \Big\}
\]
On the other hand we know that:
\[
\vert A_{i_0}\vert \in \mathbb{A} = \Big \{ \vert A_1\vert, \vert A_2\vert,...,\vert A_{i_0}\vert \Big \}\;.
\]
Therefore, the set $\mathbb{A}$ has one more element than $\mathbb{B}$ does. Hence, there exists an element  $|A_{i_1}| \in \mathbb{A}$, which is associated with an element $|B_{i_1}|$ from a complement of the set $\mathbb{B}$:
\[
|B_{i_1}| \in \mathbb{B}^c = \Big\{\vert B \vert_{(i_0)},\vert B \vert_{(i_0+1)},...,\vert B \vert_{(p)} \Big\}
\]
In consequence, due to facts $|A_{i_0}| \leqslant |A_{i_1}|$ and $|B_{i_1}| \leqslant \vert B \vert_{(i_0)}$ we obtain:
\[
 f= |A_{i_0}| - \vert B \vert_{(i_0)} \leqslant
 |A_{i_1}| - |B_{i_1}|\leqslant e\;\;,
\]
which completes the proof.\\
The proof for the case when $\vert A_{i_0} \vert < \vert B \vert_{(i_0)}$ is analogous.
\end{proof}
Now let us recall the relation between $T^{(i)}$ and $M^{(i)}$:\\
\[T^{(i)}_j:=
\Bigg\lbrace\begin{array}{l}
M^{(i)}_j + \Gamma_j \;\;\; i\neq j \\
M^{(i)}_j   \;\;\; i=j\end{array}
\]
Furthermore recall that $Q_2 = \{\max_i\vert \Gamma_i \vert \leqslant\gamma \}
$. 
When we assume that $Q_2$ occurs and apply Proposition \ref{prop2} to a vector $(\vert T^{(i)}\vert_{(1)},...,\vert T^{(i)}\vert_{(p)})'$, we obtain that:

\begin{equation}
\vert M^{(i)} \vert_{(j)} -\gamma \leqslant \vert T^{(i)} \vert_{(j)} \leqslant \vert M^{(i)} \vert_{(j)} +\gamma
\label{cond:WjMj}
\end{equation}
for all $j = 1,...,p$.
Let us recall the definitions of $H_r$ and $H^{\gamma}_r$:
\[H_r = \{
w \in \mathbb{R}^p: \forall_{j \leqslant r} \sum_{l=j}^r \lambda_l < \sum_{l=j}^r \vert w \vert_{(l)} \;\;\; and \;\;\; \forall_{j \geqslant r+1} \sum_{l=r+1}^j\lambda_l \geqslant \sum_{l=r+1}^j \vert w\vert_{(l)}
\}
\]
\[H_r^{\gamma} = \{
w \in \mathbb{R}^p: \forall_{j \leqslant r} \sum_{i=j}^r (\lambda_i- \gamma) < \sum_{i=j}^r \vert w \vert_{(i)} \;\;\; and \;\;\; \forall_{j \geqslant r+1} \sum_{i=r+1}^j(\lambda_i+ \gamma) \geqslant \sum_{i=r+1}^j \vert w\vert_{(i)}\;\;.
\}\]
It can be noticed that due to condition (\ref{cond:WjMj}) we have:
\[
T^{(i)} \in H_r \land Q_2 \Rightarrow M^{(i)} \in H^{\gamma}_r
\]
and the proof of Lemma \ref{T_to_M} is completed.
\end{proof}

\begin{proof}\textit{The proof of Lemma \ref{Conv}}\\
Without the loss of generality, let us assume that $1 \in S^c$.  
Observe that conditionally on $\epsilon$, the vector $M^{(i)}$ and the variable $M_i$ are independent. Recall that $Q_3 = \{ \frac{\Vert \epsilon \Vert_2}{\sigma\sqrt{n}} \leqslant 1+1/k^* \}$ depends only on $\epsilon$. Therefore:
\begin{equation}
P\left(M^{(i)} \in H_r^{\gamma},\vert M_i \vert> \lambda_r-\gamma,Q_3\right) = E \left( P\left(M^{(i)} \in H_r^{\gamma},Q_3| \epsilon) P(\vert M_i \vert> \lambda_r-\gamma,Q_3| \epsilon \right)  \right)\;.
\label{pom2}
\end{equation}
Now for $i \in S^c$:
\[
P\left(\vert M_i \vert> \lambda_r-\gamma,Q_3| \epsilon\right) = 
P\left(\frac{\vert X_i'\epsilon\vert}{\frac{\Vert \epsilon \Vert_2}{\sigma\sqrt{n}}}>\frac{\lambda_r-\gamma}{\frac{\Vert \epsilon \Vert_2}{\sigma\sqrt{n}}},Q_3|\epsilon\right) \leqslant\]
\begin{equation}
\leqslant
P\left(\sqrt{n}\frac{\vert X_i'\epsilon\vert}{\Vert \epsilon \Vert_2}>\frac{\lambda_r-\gamma}{\sigma(1+1/k^*)}| \epsilon\right) =
 2 \left(1-\Phi\left(\frac{\lambda_r-\gamma}{\sigma(1+1/k^*)}\right) \right)\;, 
\label{pom1}
\end{equation}
where the last equality is a consequence of the fact that conditionally on $\epsilon$, $\sqrt{n}\frac{X_i'\epsilon}{\Vert \epsilon \Vert_2}$ has a standard normal distribution. Furthermore, from the definition of $\lambda_r$ and $\gamma$ we know that for a large enough $n$
\[
\frac{\lambda_r-\gamma}{\sigma(1+1/k^*)} =   \frac{ 1+\delta- \sqrt{\frac{(k^*)^2 \log p}{n}}C_q\frac{\lambda^{BH}_{k^*}}{\lambda^{BH}_r}}{1+1/k^*}\lambda^{BH}_r \geqslant   \left(1+3\delta/4\right)\lambda^{BH}_r \geqslant (1+\delta/2)\sqrt{2 \log(p/qr)}\;,
\]
where in the last inequality we used the fact that $\lambda^{BH}_r/\sqrt{2 \log(p/qr)} \rightarrow 1$. Let us denote
\[
s = (1+\delta/2)\sqrt{2 \log(p/qr)}\;\;.
\]
By applying the above inequality to (\ref{pom1}) we obtain for a large enough $n$:
\begin{equation}
P\left(\vert M_i \vert> \lambda_r-\gamma,Q_3| \epsilon\right) \leqslant
 2 \left(1-\Phi(s) \right)\leqslant \exp(-s^2/2) = \left(\frac{qr}{p} \right)^{(1+\delta/2)^2}\;,
\label{Mi_bound}
\end{equation}
where in the second inequality we used the  classical approximation to the tail probability of the standard normal distribution (\ref{clas}).

By applying (\ref{Mi_bound}) and (\ref{pom2}) to 
\[\sum_{r=1}^{k^*}\frac{1}{r}\sum_{i \in S^{c}}P\left(M^{(i)} \in H_r^{\gamma},\vert M_i \vert> \lambda_r-\gamma,Q_3\right)\;, \]
we obtain for a large enough $n$:
\[
\sum_{r=1}^{k^*}\frac{1}{r}\sum_{i \in S^{c}}P\left(M^{(i)} \in H_r^{\gamma},\vert M_i \vert> \lambda_r-\gamma,Q_3\right) 
\leqslant
\sum_{r=1}^{k^*}\frac{1}{r}\sum_{i \in S^{c}}P\left(M^{(i)} \in H_r^{\gamma},Q_3\right) \left(\frac{qr}{p} \right)^{(1+\delta/2)^2} \leqslant
\]
\[
\leqslant
\frac{q}{p}\left(\frac{k^*}{p} \right)^{\delta} \sum_{r=1}^{k^*}\sum_{i \in S^{c}}P\left(M^{(i)} \in H_r^{\gamma},Q_3\right) = 
\frac{(p-k)q}{p}\left(\frac{k^*}{p} \right)^{\delta} \sum_{r=1}^{k^*}P\left(M^{(1)} \in H_r^{\gamma},Q_3\right)\;,\]
where in the last equality we used the assumption that $1 \in S^c$, the fact that for $i \in S^c$ all $M^{(i)}$ have the same distribution  and the fact that there are $p-k$ elements in $S^c$.\\

Due to the fact that $(k^*/p)^{\delta} \rightarrow 0$ , in order to prove the Lemma it remains to be shown that
\begin{equation}
\sum_{r=1}^{k^*}P\left(M^{(1)} \in H_r^{\gamma},Q_3\right)  \mbox{ is bounded.}
\label{pom3}
\end{equation}

We  prove this by showing that
\begin{equation}
\sum_{r=1}^{k}P( M^{(1)} \in H^{\gamma}_r,Q_3) \mbox{ is bounded}
\label{Part1}
\end{equation}
and
\begin{equation}
\sum_{r=k+2}^{k^*}P( M^{(1)} \in H^{\gamma}_r,Q_3) \rightarrow 0\;.
\label{Part2}
\end{equation}
Naturally, $P( M^{(1)} \in H^{\gamma}_{k+1}) \leqslant 1$ which, together with (\ref{Part1}) and (\ref{Part2}) would provide (\ref{pom3}) and in consequence the thesis.

We begin by proving (\ref{Part1}).
Let us denote by $\tilde{W}$ a vector of elements of $M^{(1)}$ with indices in $S$ (corresponding to non-zero elements in $b^0_i$). 
Directly from the definition of the set $H_r^{\gamma}$, we have for $r \leqslant k$:
\[
P( M^{(1)} \in H^{\gamma}_r,Q_3) \leqslant 
P\left(\vert M^{(1)} \vert_{(r+1)}\leqslant \lambda_{r+1}+\gamma,Q_3\right)\;. 
\]
Furthermore
\[
P\left(\vert M^{(1)} \vert_{(r+1)}\leqslant \lambda_{r+1}+\gamma,Q_3\right) 
\leqslant 
P\left(\vert \tilde{W} \vert_{(r)}\leqslant \lambda_{r+1}+\gamma,Q_3\right)\;,
\] 
which is a consequence of the fact that $\tilde{W}$ is a subvector of $M^{(1)}$ and that $M^{(1)}_1= \vert M^{(1)} \vert_{(1)} = \infty$ ($1 \in S^c$). Therefore, it is obvious that $\vert \tilde{W} \vert_{(r)} \leqslant \vert M^{(1)} \vert_{(r+1)}$.\\
Moreover, let $W$ be a modification of $\tilde{W}$, where each $b^0_i$ is replaced by  $2\sigma(1+\delta)\sqrt{2 \log p}$ and the resulting vector is multiplied by a function:
\[
g(\epsilon) = 
\Bigg\lbrace\begin{array}{cc}
1  & \mbox{if}\;\;Q_3\;\;\mbox{holds} \\
\infty & \mbox{otherwise}\;\;.\end{array}
\]
Therefore, $W_i = (X_j'\epsilon+ 2\sigma(1+\delta)\sqrt{2 \log p})g(\epsilon)$ for the corresponding $j \in S$. Naturally, the elements of $W$, conditionally on $\epsilon$, are independent and identically distributed.
Furthermore, due to the fact that $|X_i'\epsilon+b^0_i| =^D|X_i'\epsilon+|b^0_i||$,  the fact that the density of  $X'_i \epsilon +|b^0_i|$ has a mode at $|b^0_i|$, the assumption $\min_{i \in S} \vert b^0_i \vert > 2\sigma(1+\delta)\sqrt{2 \log p}$ and the fact that $2\sigma(1+\delta)\sqrt{2 \log p} \geqslant \lambda_{r+1}+\gamma$ for a large enough $n$, it holds that:
\[
P\left(\vert \tilde{W} \vert_{(r)}\leqslant \lambda_{r+1}+\gamma,Q_3\right)
\leqslant
P\left(\vert W \vert_{(r)}\leqslant \lambda_{r+1}+\gamma\right)\;.
\]

Now,
\[
P\left(\vert W \vert_{(r)}\leqslant \lambda_{r+1}+\gamma\right) = 
\sum_{i = k-r+1}^k P\left(exactly \;\; i \;\; of \;\; \vert W \vert 's \;\; are \;\; \leqslant \lambda_{r+1}+\gamma \right) =
\]
\[
\sum_{i = k-r+1}^k E\left[ P\left(exactly \;\; i \;\; of \;\; \vert W \vert 's \;\; are \;\; \leqslant \lambda_{r+1}+\gamma | \epsilon\right)\right] =
\]

\begin{equation}
 = 
\sum_{i = k-r+1}^k {k \choose i} E\left[ P\left(\vert W_1 \vert \leqslant \lambda_{r+1} +\gamma| \epsilon\right)^i \left(1-P\left(\vert W_1 \vert \leqslant \lambda_{r+1} +\gamma| \epsilon\right)\right)^{k-i}\right]\leqslant
\label{OrderStat}
\end{equation}
\[
 \leqslant
\sum_{i = k-r+1}^k {k \choose i} E\left[ P\left(\vert W_1 \vert \leqslant \lambda_{r+1} +\gamma| \epsilon\right)^i \right]\;,
\]
where in the equality between the second and the third line we used the standard combinatorial arguments for calculating the cumulative distribution function of the $r$-th order statistic and the fact that elements of $W$ are independent conditionally on $\epsilon$ . In the inequality we used the fact that 
$ P\left(\vert W_1 \vert \leqslant \lambda_{r+1} +\gamma| \epsilon\right) \leqslant 1$.

Now, for some $j \in S$ (corresponding to $W_1$) and a large enough $n$ we have
\[
P\left(\vert W_1 \vert \leqslant \lambda_{r+1} +\gamma| \epsilon\right)  = 
P\left(\vert (X_j'\epsilon+ 2\sigma(1+\delta)\sqrt{2 \log p})g(\epsilon) \vert \leqslant \lambda_{r+1} +\gamma| \epsilon\right)  = \]

\[
=P\left(\vert (X_j'\epsilon+ 2\sigma(1+\delta)\sqrt{2 \log p}) \vert \leqslant \lambda_{r+1} +\gamma, Q_3| \epsilon\right) \leqslant
P\left( X_j'\epsilon \leqslant -\sigma(1+\delta/2)\sqrt{2 \log p}, Q_3| \epsilon\right)\;, 
\]
where in the second line we first use the definition of the function $g(\epsilon)$ and in the inequality we skip the absolute value and use the fact that for a large enough $n$, $\lambda_{r+1} +\gamma \leqslant \sigma(1+1.5\delta)\sqrt{2 \log p}$.
Now, due to the fact that conditionally on $\epsilon$, $X_j' \epsilon$ is normal with a mean $0$ and a standard deviation $\Vert \epsilon \Vert_2/\sqrt{n}$, and the fact that $Q_3$ provides the upper bound on this standard deviation, for a large enough $n$ we obtain:
\[
P\left( X_j'\epsilon \leqslant -\sigma(1+\delta/2)\sqrt{2 \log p}, Q_3| \epsilon\right)  = 
\]
\[
= P\left( \frac{X_j'\epsilon}{\Vert \epsilon \Vert_2/\sqrt{n}}  \leqslant -\sigma \frac{(1+\delta/2)\sqrt{2 \log p}}{\Vert \epsilon \Vert_2/\sqrt{n}}, \{ \frac{\Vert \epsilon \Vert_2}{\sigma\sqrt{n}} \leqslant 1+1/k^* \}| \epsilon\right)  \leqslant
\]
\[
\leqslant P\left( \frac{X_j'\epsilon}{\Vert \epsilon \Vert_2/\sqrt{n}}  \leqslant -\frac{(1+\delta/2)\sqrt{2 \log p}}{1+1/k^*}| \epsilon\right)  =
\]
\[  = 
1 - \Phi \left( \frac{(1+\delta/2)\sqrt{2 \log p}}{(1 +1/k^*)} \right)
\leqslant \left(\frac{1}{p}\right)^{(1+\delta/4)^2}\;.
\]

In consequence we can limit (\ref{Part1}) from the above 
\[
\sum_{r=1}^{k}P( M^{(1)} \in H^{\gamma}_r) \leqslant
\sum_{r=1}^{k} \sum_{i = k-r+1}^k {k \choose i}  \left(\left(\frac{1}{p}\right)^{(1+\delta/4)^2}
\right)^i = 
\]
\[
=\sum_{r=1}^{k} r {k \choose r} \left( \left(\frac{1}{p}\right)^{(1+\delta/4)^2}\right)^r
 \leqslant
\sum_{r=1}^{k} \left(\frac{1}{r}\right)^{(r-1)}  \left( \frac{ek}{p}\left(\frac{1}{p}\right)^{(1+\delta/4)^2-1}
\right)^r \rightarrow 0
\]
for $n \rightarrow \infty$.
To see the equality between the first and the second line observe that only index $i$ depends on $r$. Therefore, we sum multiple times the same elements and for a given $i$, the summation element occurs $i$ times. The inequality uses the upper bound on the Newton symbol ${k \choose r} \leqslant \left(\frac{ke}{r}\right)^r$.
Naturally, this result is much stronger than the relation (\ref{Part1}). 

\begin{remark}
The proof of the relation (\ref{Part1}) is the only place where we use the assumption on the signal strength in the context of FDR control. Naturally, when $k$ is bounded we obtain  ( \ref{Part1}) immediately:
\[
\sum_{r=1}^{k}P( M^{(1)} \in H^{\gamma}_r,Q_3) \leqslant k
\]
and the condition on the signal strength is redundant. 
\end{remark}
Now, we turn our attention to (\ref{Part2}). From now on we assume that $r\geqslant k+2$. Let us denote by $V$ a subvector of $M^{(1)}$ consisting of elements with indices in $S^c \setminus \{1 \}$ (corresponding to elements of $b^0_i$ equal to 0, except $M^{(1)}_1$). Notice that $\vert M^{(1)} \vert_{(k+2)} \leqslant \vert V \vert_{(1)}$. This is a consequence of the fact that $\vert V \vert_{(1)}$ is the largest element in $\vert V \vert$ and is equal or larger than $p-k-1$ elements in $\vert M^{(1)}\vert$. 
Similarly, we can show that:
\begin{equation}
P( M^{(1)} \in H^{\gamma}_r,Q_3) \leqslant 
P\left(\vert M^{(1)} \vert_{(r)}\geqslant \lambda_{r}-\gamma,Q_3\right) \leqslant 
P\left(\vert V \vert_{(r-k-1)}\geqslant \lambda_{r}-\gamma,Q_3\right)\;.
\label{pom11} 
\end{equation}
Conditionally on $\epsilon$, elements of $V$ are independent and identically distributed. Therefore we have:
\[
V =^{D}\frac{\Vert \epsilon \Vert_2}{\sqrt{n}}(Z_1,...,Z_{p-k-1})'\;,
\]
where the random vector $Z = (Z_1,...,Z_{p-k-1})' \sim N(0,\mathbb{I}_{p-k-1})$ is independent of $\Vert \epsilon \Vert_2$.\\
Notice that the probability in (\ref{pom11})  is maximized for the largest possible standard deviation of $V$ (maximal $\frac{\Vert \epsilon \Vert_2}{\sqrt{n}}$). Therefore, when restricting to $Q_3$, we obtain:
\[
P\left(\vert V \vert_{(r-k-1)}\geqslant \lambda_{r}-\gamma, Q_3\right)=
P\left(\frac{\Vert \epsilon \Vert_2}{\sqrt{n}}\vert Z \vert_{(r-k-1)}\geqslant \lambda_{r}-\gamma,Q_3\right)
 \leqslant \]
 \[
 \leqslant
P\left(\vert Z \vert_{(r-k-1)}\geqslant \frac{\lambda_{r}-\gamma}{\sigma(1+1/k^*)}\right)\] 
Moreover for a large enough $n$ we have
\[
(1+\delta/2)\sqrt{2 log p} \leqslant \frac{\lambda_{r}-\gamma}{\sigma(1+1/k^*)}\;, 
\]
which follows directly from the fact that for a large enough $n$:
\[
\frac{(\lambda_{r}-\gamma)/(1+1/k^*)}{\sigma\sqrt{2 \log(p/r)}}  \geqslant (1+3\delta/4)
\]
and from the fact that:
\[
\frac{\log(p/r)}{ \log p} = 1- \frac{\log r}{ \log p}\rightarrow 1
\]
for $r \leqslant k^*$.\\
In consequence we obtain for a large enough $n$
\[
P\left(\vert Z \vert_{(r-k-1)}\geqslant \frac{\lambda_{r}-\gamma}{\sigma(1+1/k^*)}\right) \leqslant
P\left(\vert Z \vert_{(r-k-1)}\geqslant s\right) 
\] 
where
$
s = (1+\delta/2)\sqrt{2 log p}
$.\\
Let us denote by $u_1,...,u_{p-k-1}$ i.i.d. random variables from the uniform distribution $U[0,1]$ and by $u_{[1]} \leqslant... \leqslant u_{[p-k-1]}$, the corresponding order statistics.
We know that
\[
\vert Z \vert_{(r-k-1)} =^D \Phi^{-1}(1 - u_{[r-k-1]}/2)\;.
\]
Therefore, by using the classical upper bound 
\begin{equation}\label{upper}
\mbox{for}\;s>0\;,\;\;1-\Phi(s)\leq \frac{\phi(s)}{s}
\end{equation}
we obtain for a large enough $n$:
 \[
P\left(\vert Z \vert_{(r-k-1)}\geqslant s\right) =
P\left(2(1-\Phi(s))\geqslant u_{[r-k-1]} \right) \leqslant
P\left(\left(\frac{1}{p} \right)^{(1+\delta/2)^2}\geqslant u_{[r-k-1]} \right)\;. 
\] 
We also know that the i-th order statistic of the uniform distribution is a beta-distributed random variable:
\[
u_{[i]} \sim Beta(i,p-k-i)\;.
\]
On the other hand, a well known fact is that if $A_1 \sim Gamma(i,\theta)$ and $A_2 \sim Gamma(p-k-i,\theta)$ are independent, then $A_1/(A_1+A_2) \sim Beta(i,p-k-i)$.
In consequence, when $E_1,...,E_{p-k}$ are i.i.d. random variables from the exponential distribution with a mean equal to 1 it holds:
\[
u_{[i]} =^D \frac{\sum_{j = 1}^i E_j}{\sum_{j = 1}^{p-k} E_j}\;. 
\]  
Therefore
\[
P\left(\left(\frac{1}{p} \right)^{(1+\delta/2)^2}\geqslant u_{[r-k-1]} \right) 
=P\left((p-k)\left(\frac{1}{p} \right)^{(1+\delta/2)^2}\geqslant \frac{\sum_{j = 1}^{r-k-1} E_j}{\frac{1}{p-k}\sum_{j = 1}^{p-k} E_j}  \right) 
\]
and\[
P\left((p-k)\left(\frac{1}{p} \right)^{(1+\delta/2)^2}\geqslant \frac{\sum_{j = 1}^{r-k-1} E_j}{\frac{1}{p-k}\sum_{j = 1}^{p-k} E_j}  \right) 
\leqslant
P\left(2\left(\frac{1}{p} \right)^{\delta}\geqslant \sum_{j = 1}^{r-k-1} E_j \right) 
+P\left(\frac{1}{p-k}\sum_{j = 1}^{p-k} E_j>2\right)\]
\[
=F_E(2 p^{-\delta}, r-k-1,1) +P\left(\frac{1}{p-k}\sum_{j = 1}^{p-k} E_j>2\right)\;,
\]
where $F_E$ is the Erlang cumulative distribution function. Therefore we obtain that
\[
P( M^{(1)} \in H^{\gamma}_r) \leqslant 
F_E(2 p^{-\delta}, r-k-1,1) +P\left(\frac{1}{p-k}\sum_{j = 1}^{p-k} E_j>2\right)\;.
\]
To prove (\ref{Part2}) it remains to be shown that: 
\begin{equation}
\sum_{r = k+2}^{k^*}F_E(2 p^{-\delta}, r-k-1,1) \rightarrow 0
\label{pomoc1234}
\end{equation}
and
\begin{equation}
\sum_{r = k+2}^{k^*}P\left(\frac{1}{p-k}\sum_{j = 1}^{p-k} E_j>2\right) = (k^*-k-1)P\left(\frac{1}{p-k}\sum_{j = 1}^{p-k} E_j>2\right) \rightarrow 0\;.
\label{pom12}
\end{equation}
The first relation follows directly from the properties of the Erlang cumulative distribution function.
The second relation is a consequence of Chebyshev's inequality (for details see the supplementary materials).
This ends the proof.
\end{proof}

\begin{proof} \textit{Proof of Lemma \ref{Pi}}\\
Recall that we want to show
\[
P\left(\bigcap_{i \in S}\{\hat b_i \neq 0 \}\right) \rightarrow 1
\]
We can bound the considered probability in the following way:
\[
P\left(\bigcap_{i \in S}\{\hat b_i \neq 0 \}\right) = \sum_{r = 1}^p P\left(\bigcap_{i \in S}(\{\hat b_i \neq 0 \}\cap \{R = r\})\right) = 
\sum_{r = 1}^p P\left(\bigcap_{i \in S}(\{|T_i|> \lambda_r \}\cap \{R = r\})\right) \geqslant
\]
\[
\geqslant \sum_{r = 1}^p P\left(\bigcap_{i \in S}(\{|T_i|> \lambda_1 \}\cap \{R = r\})\right) = P\left(\bigcap_{i \in S}\{|T_i|> \lambda_1 \} \right)\;, 
\]
where in the first and the last equation we used the law of total probability ($\{R = r\}_{r = 1}^p$ is a partition of a sample space). The second equation is a consequence of Theorem \ref{Lem1} and the inequality comes from the fact that $\lambda$ is a decreasing sequence. Now, observe that:
\[
P\left(\bigcap_{i \in S}\{|T_i|> \lambda_1 \} \right)  \geqslant P\left(\bigcap_{i \in S}\{|T_i|> \lambda_1 \} \cap Q_2 \right) 
 \]
and recall that $T_i = M_i +\Gamma_i$ and $Q_2 = \{\max_i |\Gamma_i|\leqslant \gamma\}$. Therefore, due to the triangle inequality we have:
\[
P\left(\bigcap_{i \in S}\{|T_i|> \lambda_1 \} \cap Q_2 \right) \geqslant
P\left(\bigcap_{i \in S}\{|M_i|> \lambda_1 + \gamma \} \cap Q_2 \right)
\]
Now, because $P(Q_2) \rightarrow 1$, we only have to show that
\[
P\left(\bigcap_{i \in S}\{|M_i|> \lambda_1 + \gamma \}  \right) \rightarrow 1\;.
\]
Let us consider the properties of $M_i = X_i'\epsilon +b^0_i$. Notice that due to the symmetry of $X_i'\epsilon$ distribution, we have $|X_i'\epsilon +b^0_i| =^D |X_i'\epsilon +|b^0_i||$, therefore
\[
P\left(\bigcap_{i \in S}\{|M_i|> \lambda_1 + \gamma \}  \right) = 
P\left(\bigcap_{i \in S}\{|X_i'\epsilon +|b^0_i||> \lambda_1 + \gamma \}  \right) \geqslant
P\left(\bigcap_{i \in S}\{X_i'\epsilon > \lambda_1 + \gamma -|b^0_i| \}  \right)\;, 
\]
where in the last inequality we omit the absolute value in $|X_i'\epsilon +|b^0_i||$ and subtract $|b^0_i|$. Now, due to the assumption on the signal strength, we have for a large enough $n$:
\[
P\left(\bigcap_{i \in S}\{X_i'\epsilon > \lambda_1 + \gamma -|b^0_i| \}  \right) \geqslant
P\left(\bigcap_{i \in S}\{X_i'\epsilon > -\sigma(1+\delta/2)\sqrt{2 \log p} \}  \right)\;. 
\]
This is a consequence of the fact that for a large enough $n$:
\[
\lambda_1 + \gamma \leqslant\sigma(1+1.5\delta)\sqrt{2 \log p}\;.
\]
Moreover, we know that conditionally on $\epsilon$, $X'_i\epsilon$ are independent from the normal distribution $N(0, \Vert \epsilon \Vert_2^2/n) $. Therefore we have:
\[
P\left(\bigcap_{i \in S}\{X_i'\epsilon > -\sigma(1+\delta/2)\sqrt{2 \log p} \}  \right) \geqslant
P\left(\bigcap_{i \in S}\{X_i'\epsilon > -\sigma(1+\delta/2)\sqrt{2 \log p} \}\cap Q_3  \right) =
\]
\[
 = P\left(\bigcap_{i \in S}\{\sqrt{n}\frac{X_i'\epsilon}{\Vert \epsilon \Vert_2} > -\sigma\sqrt{n}\frac{(1+\delta/2)\sqrt{2 \log p}}{\Vert \epsilon \Vert_2} \}\cap Q_3  \right)\geqslant
 \]
 \[
 \geqslant
P\left(\bigcap_{i \in S}\{\sqrt{n}\frac{X_i'\epsilon}{\Vert \epsilon \Vert_2} > -\frac{(1+\delta/2)}{(1+1/k^*)}\sqrt{2 \log p} \}\cap Q_3  \right)\;,
\]
where in the last inequality we have used the condition defining $Q_3$. 
Again, due to the fact that $P(Q_3) \rightarrow 1$, we only have to show that
\[
P\left(\bigcap_{i \in S}\{\sqrt{n}\frac{X_i'\epsilon}{\Vert \epsilon \Vert_2} > -a\sqrt{2 \log p} \}  \right) \rightarrow 1\;,
\]
where $a = \frac{(1+\delta/2)}{(1+1/k^*)}$.
Now, using the fact that conditionally on $\epsilon$, $X'_i\epsilon$ are independent random variables  from the normal distribution $N(0, \Vert \epsilon \Vert_2^2/n)$, we obtain
\[
P\left(\bigcap_{i \in S}\{\sqrt{n}\frac{X_i'\epsilon}{\Vert \epsilon \Vert_2} > -a\sqrt{2 \log p} \}  \right) =
E\left(P\left(\bigcap_{i \in S}\{\sqrt{n}\frac{X_i'\epsilon}{\Vert \epsilon \Vert_2} > -a\sqrt{2 \log p} \}| \epsilon  \right) \right) = 
\]
\[
=E\left(\prod_{i \in S}P\left(\{\sqrt{n}\frac{X_i'\epsilon}{\Vert \epsilon \Vert_2} > -a\sqrt{2 \log p} \}| \epsilon  \right) \right) = 
E\left[\prod_{i \in S}\left(1-\Phi \left( -a\sqrt{2 \log p} \right) \right) \right] = 
\]
\[
=\left[1-\left(1 -\Phi \left(a\sqrt{2 \log p} \right) \right)\right]^k \geqslant
\left[1-exp\left(-0.5 (a\sqrt{2 \log p})^2\right)  \right]^k = 
\left[1-\left( \frac{1}{p}\right)^{a^2}  \right]^k \geqslant
\]
\[ \geqslant 1 - k \left( \frac{1}{p}\right)^{a^2} \rightarrow 1\;,
\]
where in the above inequalities we used the bound (\ref{upper}) and Bernoulli's inequality. The convergence is a consequence of the fact that for a large enough $n$, $a \geqslant 1$ and the assumption that $k/p \rightarrow 0$. This ends the proof of $\Pi \rightarrow 1$.
\end{proof}

\bibliographystyle{plainnat}
\bibliography{bibliography}

\section{Supplementary materials}

\subsection{Proofs of Theorems \ref{Lem2} and \ref{Lem1}}

In this Section we present proofs of Theorems \ref{Lem2} and \ref{Lem1} and some additional interesting facts. The proofs are similar to each other on certain level. Therefore we organize them in a way that in authors opinion will be the most convenient to the reader.
We would like to start with presenting some facts showing a specific connection between the score vector $U_i(\hat{b})$ and the estimator $\hat{b}$.\\
Straightforward, from the proof of the Theorem \ref{Lem1}, we obtain following fact:
\begin{corollary}
\label{biUi}
Suppose assumptions of the Theorem \ref{Lem1} holds. When $\hat{b}_i \neq 0$ then $U_i(\hat{b})$ and $\hat{b}_i$ have the same sign (see inequality \ref{ineq:sign}).
\end{corollary}

\begin{proposition}
\label{Ui_order}
Consider optimization problem given by generalized SLOPE (see: $\ref{gSLOPE}$ in the Article) with an arbitrary sequence $\lambda_1\geqslant \lambda_2 \geqslant ... \geqslant \lambda_p \geqslant 0$ and assume that $l(b)$ is a convex and differentiable function.\\
Under above assumptions if $\vert \hat{b}_j \vert > \vert \hat{b}_k \vert$ then $ \vert U_j(\hat{b}) \vert \geqslant \vert U_k(\hat{b}) \vert$.
\end{proposition}
Above Proposition enable further specification of the $(|b_1|,...,|b_p|)'$ ordering in case when $\vert b \vert_{(j)} = \vert b \vert_{(j+1)}$ for some $j$.
Let's introduce a notation $U_{(j)^\#}(b)$ and $b_{(j)^\#}$ on the score statistic and element of the vector $b$ associated with the element $ \vert b \vert_{(j)}$. Of course when $\vert b \vert_{(j)} = \vert b \vert_{(j+1)}$ the indexing is ambiguous and therefore we define it in a way that associated statistics have a following property:
 \[
\vert  U_{(j)^\#}(b) \vert \geqslant \vert U_{(j+1)^\#}(b) \vert
 \]

\begin{remark}
\label{remark}
By the above ordering and the Proposition \ref{Ui_order} we know that the vector $(\vert  U_{(1)^\#}(\hat{b}) \vert,...,\vert U_{(p)^\#}(\hat{b}) \vert)'$ associated with $(\vert\hat{b}\vert_{(1)},...,\vert \hat{b}\vert_{(p)})'$ is ordered $(\vert  U_{(1)^\#}(\hat{b}) \vert \geqslant ... \geqslant \vert U_{(p)^\#}(\hat{b}) \vert)$ and in consequence we can write $\vert  U_{(j)^\#}(\hat{b}) \vert = \vert  U(\hat{b}) \vert_{(j)}$.
\end{remark}

\begin{corollary}
\label{coro2}
Consider optimization problem given by generalized SLOPE (see: $\ref{gSLOPE}$ in the Article) with an arbitrary sequence $\lambda_1\geqslant \lambda_2 \geqslant ... \geqslant \lambda_p \geqslant 0$ and assume that $l(b)$ is a convex and differentiable function. Then for any arbitrary $a>0$ and $j = 1,...,p$, it holds:
\[
\vert  U_{(j)^\#}(\hat{b}) + a \hat{b}_{(j)^\#}\vert =
 \vert U(\hat{b}) + a \hat{b} \vert_{(j)} =  \vert T(a) \vert_{(j)}
\]
\end{corollary}

\begin{proof}
Observe that:
\[
\vert  U_{(j)^\#}(\hat{b}) + a \hat{b}_{(j)^\#}\vert = 
\vert  U_{(j)^\#}(\hat{b}) \vert + \vert a \hat{b}_{(j)^\#}\vert = 
\vert U(\hat{b})\vert_{(j)} + \vert a \hat{b} \vert_{(j)}=
\vert U(\hat{b}) + a \hat{b} \vert_{(j)} =  \vert T(a) \vert_{(j)}
\] 
The first equality is a consequence of Corollary \ref{biUi}. The second results from the Remark \ref{remark}. The third one comes from a fact that sum of two positive ordered sequences is also ordered and the last is simply a consequence of the $T(a)$ definition. 
\end{proof}

Through out this section we shall denote the objective function of generalized SLOPE (see: $\ref{gSLOPE}$ in the Article) by:
\[
f(b) = l(b) + \sum_{j=1}^p \lambda_{j} \vert b \vert_{(j)}
\]
\begin{proof} The proof of the Theorem \ref{Lem1}.\\
We will start the proof by showing that:
\[
\left( \hat{b}_i \neq 0 \right)
\Rightarrow \left(\vert T_i(a) \vert >\lambda_r\right) 
\]
and
\[
\left( \hat{b}_i \neq 0 \right)
\Rightarrow \left(\vert U_i(\hat b) \vert \geqslant\lambda_r\right) 
\]

From the optimality of $\hat{b}$ we have following inequality for any vector $b$: 
\begin{equation}
0\geqslant f(\hat{b})-
f(b)= l(\hat{b})-l(b)+ \sum_{j=1}^p \lambda_{j} \vert \hat{b}\vert_{(j)} -
 \sum_{j=1}^p \lambda_{j} \vert b \vert_{(j)}
 \label{ineq:main_A.2}
\end{equation}
In the first step we will prove that for a vector $b=\hat{b}+h(0,...,0,\underbrace{(-1) sgn(\hat{b}_i)}_{i-th \;\; pos.},0,...,0)^T=\hat{b}+hL_i$ and small enough, positive $h$ we have:
\begin{equation}
\sum_{j=1}^p \lambda_{j} \vert \hat{b}\vert_{(j)} -
 \sum_{j=1}^p \lambda_{j} \vert b \vert_{(j)} \geqslant h\lambda_r
\label{ineq:A.2.1}
\end{equation}
Let's consider first the simplest scenario where $\vert \hat{b}_i \vert \neq \vert \hat{b}_j \vert$ for any $i \neq j$.
In this situation if we take 
\[h<\min_{j,\;\;i\neq j}\left(\left\vert  \vert\hat{b}_i\vert - \vert\hat{b}_j\vert \right\vert\right)\] 
then the ordering of an absolute value of the elements in the vector $b$ and $\hat{b}$ will be the same. Furthermore these vectors differ only in the i-th position and from the form of a vector $b$ we know that the absolute value of an i-th element in vector $b$ is smaller then in vector $\hat{b}$. In consequence we obtain for certain index $k$ that: 
\[
\sum_{j=1}^p \lambda_{j} \vert \hat{b}\vert_{(j)} -
 \sum_{j=1}^p \lambda_{j} \vert b \vert_{(j)} = h\lambda_k \geqslant h\lambda_r
\]
The last inequality is a consequence of assumptions $\#\{i : \hat{b}_i \neq 0\} = r$ and $\hat{b}_i \neq 0$ and a fact that $\lambda_r$ is the smallest element of vector $\lambda$ associated with non-zero elements of the vector $\hat{b}$.\\
Let's consider now more general scenario where two or more elements of the vector $\hat{b}$ have the same absolute value as $\vert \hat{b}_i \vert$. The reasoning is analogous, however in this situation value $\vert \hat{b}_i \vert$ will be associated with a set of elements of vector $\lambda$. Of course the power of associated set will be the same as the number of elements of vector $\hat{b}$ with absolute value the same as $\vert \hat{b}_i \vert$. Moreover the associated set of lambdas will contain consecutive elements of whole sequence. It is also straightforward that the indexing in a group of elements of vector $\hat{b}$ with the same module as $\vert \hat{b}_i \vert$ is ambiguous. For such situation if we take
\[h<\min_{j,\;\vert\hat{b}_i\vert \neq \vert\hat{b}_j\vert}\left(\left\vert  \vert\hat{b}_i\vert - \vert\hat{b}_j\vert \right\vert\right)\] 
we obtain for certain index $k$ that: 
\[
\sum_{j=1}^p \lambda_{j} \vert \hat{b}\vert_{(j)} -
 \sum_{j=1}^p \lambda_{j} \vert b \vert_{(j)} = h\lambda_k \geqslant h\lambda_r
\]
however this time $\lambda_k$ is the smallest element associated with $\vert \hat{b}_{i} \vert$.
The justification for the last inequality is analogous to the previously discussed situation.
\\
Using proven relation (\ref{ineq:A.2.1}) in the inequality (\ref{ineq:main_A.2}) we obtain:
\[
\lambda_r \leqslant\dfrac{l(b)-l(\hat{b})}{h}
\]
Via transformation of an expression on the right side we obtain:
\begin{equation}
\lambda_r \leqslant \dfrac{l(b)-l(\hat{b})}{h} =\dfrac{l(\hat{b}+hL_i)-l(\hat{b})}{h}  
\rightarrow sgn(\hat{b}_i) U_i(\hat{b})
\label{ineq:sign}
\end{equation}
for $h\rightarrow 0$. 
By adding to both sides factor $a\vert \hat{b}_i \vert$ we obtain:
\[
\lambda_r<\lambda_r + a\vert \hat{b}_i \vert \leqslant  sgn(\hat{b}_i) (U_i(\hat{b})+a\hat{b}_i) \leqslant 
\vert U_i(\hat{b})+a\hat{b}_i \vert
\]
Naturally if $a = 0$ then:
\[
\lambda_r  \leqslant 
\vert U_i(\hat{b})\vert
\]

This ends the proof in one direction.\\
It remains to show that
\[
\left( \hat{b}_i \neq 0 \right)
\Leftarrow \left(\vert T_i(a) \vert >\lambda_r\right) 
\]
and that when $\lambda_1>...>\lambda_p \geqslant 0$ then
\[
\left( \hat{b}_i \neq 0 \right)
\Leftarrow \left(\vert U_i(\hat b) \vert \geqslant\lambda_r\right) 
\]

The proof of implications will be shown via contradiction. Let's assume that $ \hat{b}_i = 0$ and $\lambda_r < \vert U_i(\hat{b}) +a\hat{b}_i \vert$ and consider vector $b=\hat{b}+h(0,...,0,\underbrace{1}_{i-th \;\; pos.},0,...,0)^T=\hat{b}+hL_i$. Let's recall that $\lambda_{r+1}$ is the largest element of the sequence $\lambda$ associated with the elements of the vector $\hat{b}$ with value 0. It is easy to notice that for small enough, positive $h$ ($h<\min_i\{\vert \hat{b}_i \vert, \; \hat{b}_i \neq 0\}$) we obtain:
\[
\sum_{j=1}^p \lambda_{j} \vert \hat{b}\vert_{(j)} -
 \sum_{j=1}^p \lambda_{j} \vert b \vert_{(j)} = - h\lambda_{r+1}
\]
and in consequence following inequality holds:
\[
0\geqslant f(\hat{b})-f(b)= l(\hat b)-l(b) -h\lambda_{r+1}
\]
Transforming above inequality we obtain:
\[
\lambda_{r+1}\geqslant -\dfrac{l(\hat{b}+hL_i)-l(\hat{b})}{h} \rightarrow U_i(\hat{b})
\]
Analogical reasoning for vector $b=\hat{b}+h(0,...,0,\underbrace{-1}_{i-th \;\; pos.},0,...,0)^T$ provides us:
\[
\lambda_{r+1}\geqslant - U_i(\hat{b})
\]
In consequence we obtain that $\vert U_i(\hat{b}) \vert \leqslant \lambda_{r+1}$. On the other hand we assumed that $\lambda_r < \vert U_i(\hat{b}) + a\hat{b}_i\vert=\vert U_i(\hat{b}) \vert$ (assumpt. $\hat{b}_i=0$) which leads to contradiction and ends the proof of first equivalence.

It is easy to notice that we used the assumption $\lambda_r < \vert U_i(\hat{b}) + a\hat{b}_i\vert=\vert U_i(\hat{b}) \vert$ at the end of the above reasoning.  Therefore in the case of the second equivalence the calculation are the same and in consequence we can again use condition $\vert U_i(\hat{b}) \vert \leqslant \lambda_{r+1}$. On the other hand we assume that $\lambda_{r+1} < \lambda_{r}\leqslant \vert U_i(\hat{b}) \vert$ which again leads to contradiction and ends the proof.
\end{proof}

\begin{proof} 
The proof of Proposition \ref{Ui_order}.\\
The proof is analogical to a proof of the Theorem \ref{Lem1}.\\
Let's consider a vector $b=\hat{b}+h(0,...,0,\underbrace{(-1) sgn(\hat{b}_j)}_{j-th \;\; pos.},0,...,0,\underbrace{sgn(\hat{b}_k)}_{k-th \;\; pos.},0,...,0)^T=\hat{b}+hL $. For small enough positive h:
\[
0\geqslant f(\hat{b})-f(b)= l(\hat b)-l(b)+ h\lambda_{l}  
- h\lambda_{m}\]
where $\lambda_{l}$ and $\lambda_{m}$ are the smallest and the largest element of the sequence $\lambda$ associated with the elements $\vert\hat{b}_j\vert$ and $\vert\hat{b}_k\vert$, respectively. This is a consequence of the vector $b$ form. By its definition the $j-th$ element is pulled towards 0 and the $k-th$ element is pushed away from zero. Now, from the assumption $\vert \hat{b}_j \vert > \vert \hat{b}_k \vert$ we know that $\lambda_l \geqslant  \lambda_m$.\\
Therefore, after transformation we obtain:
\[
0 \leqslant \lambda_{l}-\lambda_m\leqslant\dfrac{l(b)-l(\hat{b})}{h} \rightarrow \triangledown_Ll(\hat{b})= sgn(\hat{b}_j)U_j(\hat{b}) - sgn(\hat{b}_k)U_k(\hat{b})
\]
for $h \rightarrow 0 $.
The last equation is a consequence of differentiability of function $l(b)$. Using relation $sgn(\hat{b}_i)U_i(\hat{b}) = \vert U_i(\hat{b}) \vert$ (see corollary \ref{biUi}) we obtain:
\[
0 \leqslant \vert U_j(\hat{b}) \vert - \vert U_k(\hat{b}) \vert
\]
This ends the proof.
\end{proof}

\begin{proof} 
The proof of the Theorem \ref{Lem2}.\\
Recall that $T(a) = U(\hat{b})+a \hat{b}$ and that from the definition of the set $H_r$ we know that $\left(U(\hat{b})+a \hat{b}\right) \in H_r$ is fulfilled if and only if both following condition are satisfied:\\
$i)$\\
\[
\forall_{j \leqslant r} \sum_{i=j}^r \lambda_i < \sum_{i=j}^r \vert U(\hat{b})+a \hat{b} \vert_{(i)}
\]
$ii)$\\
\[
\forall_{j \geqslant r+1} \sum_{i=r+1}^j\lambda_i \geqslant \sum_{i=r+1}^j \vert U(\hat{b})+a \hat{b} \vert_{(i)}
\]
Let's assume $\#\{i : \hat{b}_i \neq 0\}=r$ and determine $j \leqslant r$.\\
Furthermore, let $I$ be a set of indexes for which the rank of absolute value of the element of a vector $\hat{b}$ is between $j$ and $r$. Similarly to the proof of Theorem \ref{Lem1} we consider vector $b$ defined in a following way:
\[
b=\Bigg\lbrace\begin{array}{l}
\hat{b}_i-h \;sgn(\hat{b}_i) \;\;\; i \in I \\
\hat{b}_i \;\;\; otherwise
\end{array}
\]
Again, for $h$ positive and small enough we obtain:
\[
0\geqslant f(\hat{b})-f(b)= l( \hat b)-l(b)+h\sum_{i=j}^r \lambda_i
\]
Above is a consequence of a construction of the vector $b$. It is obtained via pulling $r-(j-1)$ smallest nonzero elements in the vector $\hat{b}$ towards zero by the same factor $h$.\\
Transforming above inequality we obtain:
\[
\sum_{i=j}^r \lambda_i\leqslant \dfrac{l( b)-l(\hat b)}{h} \longrightarrow \; \sum_{i=j}^r sgn\left(\hat{b}_{(i)^\#}\right)U_{(i)^\#}(\hat{b})=
\sum_{i=j}^r \vert U_{(i)^\#
}(\hat{b})\vert=
\sum_{i=j}^r \vert U(\hat{b})\vert_{(i)}
\]
for $h\rightarrow0$. By adding to both sides of the inequality $\sum_{i=j}^r\vert a\hat{b} \vert_{(i)}$ we obtain:
\[
\sum_{i=j}^r 
\lambda_i<\sum_{i=j}^r (\lambda_i + \vert a\hat{b} \vert_{(i)}) \leqslant 
\sum_{i=j}^r  \left(\vert U(\hat{b})\vert_{(i)} + \vert a\hat{b} \vert_{(i)} \right) = \sum_{i=j}^r \vert U(\hat{b}) + a\hat{b} \vert_{(i)}
\]
which ends the proof that $\#\{i : \hat{b}_i \neq 0\}=r$ ensure (i).\\
Now let's determine $j \geqslant r+1$ and introduce set $J$ of indexes for which the rank of an absolute value of the element of a vector $\hat{b}$ is between $r+1$ and $j$.
Let's consider following set of vectors $b$:
\[
b=\Bigg\lbrace\begin{array}{l}
hE_i  \;\;\; i \in J \\
\hat{b}_i \;\;\; otherwise
\end{array}
\]
where $E_i=\pm 1$ and $h>0$.\\
For small enough h and for any vector $b$ we obtain:
\[
0\geqslant f(\hat{b})-f(b)= l(\hat b)-l(b)- h\sum_{i=r+1}^j \lambda_i
\]
Again it is associated with a form of vector b. However this time $j-r$ elements with value 0 in the vector $\hat{b}$ are pushed away from 0 by a factor $h$.\\
By transformation of the above inequality we obtain:
\[
\sum_{i=r+1}^j \lambda_i\geqslant \dfrac{l(\hat b)-l(b)}{
h} \longrightarrow \sum_{i=r+1}^j E_{(i)^\#}U_{(i)^\#}(\hat{b})
\]
for $h \rightarrow 0$.
This inequality is valid for any sequence of $\{E_i\}$. Therefore:
\[
\sum_{i=r+1}^j\lambda_i \geqslant 
\sum_{i=r+1}^j \vert U_{(i)^\#}(\hat{b})\vert =
\sum_{i=r+1}^j \vert U(\hat{b})\vert_{(i)} =
\sum_{i=r+1}^j \vert U(\hat{b}) + a\hat{b}\vert_{(i)} 
\]
The last equality is a consequence of a fact that $\hat{b}_{(i)}=0$ for $i>r$.\\
This ends the proof that $\#\{i : \hat{b}_i \neq 0\}=r$ ensure (ii).\\
To prove the equivalence in other direction we will show that there are at least and at most $r$ nonzero elements in $\hat{b}$.\\
Suppose that conditions (i) and (ii) are fulfilled and that $\#\{i : \hat{b}_i \neq 0\}=j<r$. From the first part of the proof we know that when $\#\{i : \hat{b}_i \neq 0\}=j$ then for $m>j$ we have $\vert\hat{b}\vert_{(m)}=0$ and:
\[
\sum_{i=j+1}^r\lambda_i \geqslant \sum_{i=j+1}^r \vert U(\hat{b})\vert_{(i)}
\]
On the other hand from (i) we have:
\[
\sum_{i=j+1}^r \lambda_i < \sum_{i=j+1}^r  \vert U(\hat{b}) \vert_{(i)}
\]
which leads to a contradiction.\\
Now let's assume that $\#\{i : \hat{b}_i \neq 0\}=j>r$. This implies that:
\[
\sum_{i=r}^j
 \lambda_i < \sum_{i=r}^j \vert U(\hat{b})+ a\hat{b} \vert_{(i)}
\]
On the other hand from (ii) we obtain:
\[
\sum_{i=r}^j
 \lambda_i \geqslant \sum_{i=r}^j \vert U(\hat{b})+ a\hat{b} \vert_{(i)}
\]
which again leads to a contradiction and ends the proof of the Theorem.
\end{proof}

\subsection{Results used in the proof of the Theorem \ref{main}}
In this section we present theoretical results that are used to prove the  Theorem \ref{main}.
\begin{definition}[Resolvent set]
Fix $S = supp(b^0)$ of cardinality k, and an integer $k^*$
obeying $k < k^* < p$. The set $S^* = S^*(S, k^*)$ is said to be a resolvent set if it is the union of S and
the $k^*-k$ indexes with the largest values of 
$\vert X_i' \epsilon \vert$
 among all $i \in \{1,...,p\} \setminus S$.
\end{definition}
 
\begin{lemma}[\citet{su2016}, Lemma 4.4]\label{resolvent1}
Suppose we consider the sequence of linear models (\ref{model}) where $p \rightarrow \infty, k/p \rightarrow 0$,$(k\log p)/n \rightarrow 0$. Moreover, let
\begin{equation}
k^* \geqslant max \left(\frac{1+c}{1-q}k, k + d\right)
\label{cond:k}
\end{equation}
for an arbitrary small constant $c > 0$, and where $d$ is a deterministic sequence diverging to infinity in such a way that $k^*/p \rightarrow 0$ and $(k^* \log p)/n \rightarrow 0$. Then
\begin{equation}
P(\{Supp(b^0)\cup Supp(\hat{b})\subseteq S^*\}) \rightarrow 1.
\label{b_in_S*}
\end{equation}
\end{lemma}

\begin{proof}Proof of Corollary \ref{resolvent}.\\
Naturally, it is easy to notice that assumptions of the main Theorem in Article (\ref{main}) are stronger than those of the Lemma \ref{resolvent1}. In consequence under assumptions of the Theorem \ref{main}, we can choose a sequence $d$ and $k^*$ in such a way that \ref{b_in_S*} holds, $k^*/p \rightarrow 0$ and $((k^*)^2 \log p)/n \rightarrow 0$. 
\end{proof}

\begin{definition} 
Given two metric spaces (X, $d_X$) and (Y, $d_Y$), where $d_X$ denotes the metric on the set X and $d_Y$ is the metric on set Y, a function $f : X \rightarrow Y$ is called L - lipschitz continuous if there exists a real constant $L \geqslant 0$ such that, for all $x_1$ and $x_2$ in X,

 \[d_{Y}(f(x_{1}),f(x_{2}))\leq Ld_{X}(x_{1},x_{2})\]
 
\end{definition}

\begin{theorem}\citep{vershynin_2012}
\label{Borell}
Let $\zeta \sim N(0,\mathbb{I}_n)$ and $f$ be a L-lipschitz continous function in $\mathbb{R}^n$. Then 
\[
P(f(\zeta) \leqslant Ef(\zeta) + t) \geqslant 1 -  e^{-\frac{t^2}{2L^2}}
\]
for all $t\geqslant 0$.
\end{theorem}

We shall denote by $X_I,\hat{b}_I,b^0_I$  a submatrix (subvector) of $X ,\hat{b},b^0$ consisted of columns (vector elements) with indexes in a set $I$.

\begin{lemma}[\citet{su2016}, Lemma A.12]
\label{Gamma_cond_1}
Under assumptions of the Theorem \ref{main} (in original paper the assumptions where less stringent) there exist a constant $C_q$ only depending on q such that for all $j$ we have:
\[
\sum_{i=1}^j \vert X_{{(S^*)}^c}'X_{S^*}(\hat{b}_{S^*}-b^0_{S^*}) \vert_{(i)} 
\leqslant C_q \sqrt{\frac{k^* \log p}{n}} \sum_{i = k^* +1}^{k^*+j}\lambda^{BH}_{i}
\]
with probability tending to 1.
\end{lemma}
\begin{proof}The proof of the Corollary \ref{eq:L3.2.Sec}.\\
Directly from the Lemma \ref{Gamma_cond_1} for $j=1$ we obtain:
\begin{equation}
\max_{i \in {(S^*)}^c}\vert [X_{{(S^*)}^c}'X_{S^*}(\hat{b}_{S^*}-b^0_{S^*})]_{i} \vert =
\vert X_{{(S^*)}^c}'X_{S^*}(\hat{b}_{S^*}-b^0_{S^*}) \vert_{(1)} 
\leqslant
C_q \sqrt{\frac{k^* \log p}{n}} \lambda^{BH}_{k^* +1}
\end{equation}
with probability tending to 1.
\end{proof}

\begin{lemma}(\citet{su2016}, Lemma A.11)
\label{X_SV}
Let $k< k^*< \min\{n,p\}$ be any (deterministic) integer. Denote by $\sigma_{max}$ and $\sigma_{min}$, the largest and the smallest singular value of $X_{S^*}$. Then for any $t>0$,
\[
\sigma_{min}> \sqrt{1-1/n}-\sqrt{k^*/n} - t
\]
holds with probability at least $1 - e^{-nt^2/2}$. Furthermore,
\[
\sigma_{max}< \sqrt{1-1/n}+\sqrt{k^*/n} + \sqrt{8k^* \log(p/k^*)/n} + t
\]
holds with probability at least $1 - e^{-nt^2/2} - (\sqrt{2}ek^*/p)^{k^*}$.
\end{lemma}

\begin{proof} The proof of relations \ref{pomoc1234} and \ref{pom12}.\\
Recall the form of \ref{pomoc1234}:
\begin{equation}
\sum_{r = k+2}^{k^*}F_E(2 p^{-\delta}, r-k-1,1) \rightarrow 0
\end{equation}
Furthermore, recall that
\[
F_E(x, m,1) = 1- \exp(-x)\sum_{l = 0}^{m-1}\frac{1}{l!}(x)^l =  \exp(-x)\sum_{l = m}^{\infty}\frac{1}{l!}(x)^l
\]
Therefore
\[
\sum_{r = k+2}^{k^*}F_E(2 p^{-\delta}, r-k-1,1) = 
\exp(-2 p^{-\delta})\sum_{r = k+2}^{k^*} \sum_{l = r-k-1}^{\infty}\frac{1}{l!}(2 p^{-\delta})^l = 
\]
\[
\exp(-2 p^{-\delta})\left( \sum_{l = 1}^{\infty}\frac{1}{l!}(2 p^{-\delta})^l +\sum_{l = 2}^{\infty}\frac{1}{l!}(2 p^{-\delta})^l+...+\sum_{l = k^*-k-1}^{\infty}\frac{1}{l!}(2 p^{-\delta})^l \right) = 
\]
\[
 =\exp(-2 p^{-\delta})\left[ \sum_{l = 1}^{k^*-k-1}l\left(\frac{1}{l!}(2 p^{-\delta})^l\right) + 
\sum_{l = k^*-k}^{\infty}\underbrace{(k^*-k-1)}_{\leqslant l} \left(\frac{1}{l!}(2 p^{-\delta})^l\right)  \right] \leqslant 
\]
\[
\exp(-2 p^{-\delta})(2 p^{-\delta})\sum_{l = 1}^{\infty}\frac{1}{(l-1)!}(2p^{-\delta})^{(l-1)} = (2 p^{-\delta}) \rightarrow 0 
\]
which we wanted to show. Now, recall the form of the \ref{pom12}:
\begin{equation}
\sum_{r = k+2}^{k^*}P\left(\frac{1}{p-k}\sum_{j = 1}^{p-k} E_j>2\right) = (k^*-k-1)P\left(\frac{1}{p-k}\sum_{j = 1}^{p-k} E_j>2\right) \rightarrow 0
\end{equation}
To see that above is also valid observe that:
\[
E\left(\frac{1}{p-k}\sum_{j = 1}^{p-k} E_j\right) = 1
\]
\[
Var\left(\frac{1}{p-k}\sum_{j = 1}^{p-k} E_j\right) =\left(\frac{1}{p-k}\right)^2 \sum_{j = 1}^{p-k}Var\left( E_j\right) = \frac{1}{p-k} = \sigma_1^2
\]
Therefore do to the Chebyshev's inequality we obtain that
\[
P\left(\frac{1}{p-k}\sum_{j = 1}^{p-k} E_j>2\right) =
P\left(\frac{1}{p-k}\sum_{j = 1}^{p-k} E_j-1>1\right)  \leqslant
\]
\[
 \leqslant
P\left(\Big\vert\frac{1}{p-k}\sum_{j = 1}^{p-k} E_j-1\Big\vert>\sigma_1\sqrt{(p-k)}\right) \leqslant \frac{1}{p-k}
\]
and in consequence
\[
(k^*-k-1)P\left(\frac{1}{p-k}\sum_{j = 1}^{p-k} E_j>2\right)\leqslant \frac{k^*-k-1}{p-k} \rightarrow 0
\]
which ends the proof. 
\end{proof}

\end{document}